\documentclass[preprint,3p,times,12.8pt]{elsarticle}
\usepackage{amsthm,amsmath,amssymb}  
\usepackage{moreverb}
\usepackage{mathrsfs,bm}
\usepackage{amsfonts,graphicx,subfigure,ifpdf,enumerate}
\usepackage{lineno,hyperref,lipsum}

\usepackage{float}
\usepackage{hyperref}
\hypersetup{hypertex=true,
	colorlinks=true,
	linkcolor=blue,
	anchorcolor=blue,
	citecolor=blue}

\usepackage{lineno}

\begin{document}
	
	\newtheorem{theorem}{Theorem}[section]
	\newtheorem{lemma}[theorem]{Lemma}
	\newtheorem{assumption}[theorem]{Assumption}
	\newtheorem{proposition}[theorem]{Proposition}
	\newtheorem{corollary}[theorem]{Corollary}
	\newtheorem{definition}{Definition}[section]
	\newtheorem{remark}{Remark}[section]
	\newtheorem{example}{Example}[section]

\begin{frontmatter}
	\author{Jie Zhu\fnref{addr0}}
	\author{Yujun Zhu\fnref{addr0}} 
	\author{Ju Ming\fnref{addr0,addr2}\corref{cor1}} 
	\cortext[cor1]{Corresponding author.\\ 
		Email: jming@hust.edu.cn (Ju Ming)}
	\author{Max D. Gunzburger\fnref{addr3}}
	\address[addr0]{School of Mathematics and Statistics, Huazhong University of Science and Technology, Wuhan 430074, PR China}
	\address[addr2]{Hubei Key Laboratory of Engineering, Modeling and Scientific Computing, Huazhong University of Science and Technology, Wuhan 430074, PR China}
	
	\address[addr3]{Department of Scientific Computing, 
		Florida State University, 
		Tallahassee,  FL 32304, USA}
	
	
	\title{Splitting Method for Stochastic Navier-Stokes Equations}

	\begin{abstract}
		This paper investigates the two-dimensional stochastic steady-state  Navier-Stokes(NS) equations with additive random noise. We introduce an innovative splitting method that decomposes the stochastic NS equations into a deterministic NS component and a stochastic equation. We rigorously analyze the proposed splitting method from the perspectives of equivalence, stability,  existence and uniqueness of the solution. We also propose a modified splitting scheme, which simplified the stochastic equation by omitting its nonlinear terms.  A detailed analysis of the solution properties for this modified approach is provided. Additionally, we discuss the statistical errors with both the original splitting format and the modified scheme. Our theoretical and numerical studies demonstrate that the equivalent splitting scheme exhibits significantly enhanced stability compared to the original stochastic NS equations, enabling more effective handling of nonlinear characteristics. Several numerical experiments were performed to compare the statistical errors of the splitting method and the modified splitting method.   Notably, the deterministic NS equation in the splitting method does not require repeated solving, and the stochastic equation in the modified scheme is free of nonlinear terms. These features make the modified splitting method particularly advantageous for large-scale computations, as it significantly improves computational efficiency without compromising accuracy.
	\end{abstract}
	
	
\end{frontmatter}

\section{Introduction}\label{sec1}
The stochastic Navier-Stokes (SNS) equation constitutes a significant extension of the classical Navier-Stokes (NS) \cite{C_L_1823} equation, incorporating the inherent variability and unpredictability in real-world fluid dynamics. The SNS equation finds relevance in diverse fields, including machinery, power generation, construction, water management, chemical engineering, energy, environmental science, and biology
\cite{Przemyslaw_2019,Bhatti_2020,Yang_2021}.  In recent years, scholars have proposed a large number of methods for dealing with stochastic complex systems\cite{Holden_1996,Carmona_1999,Kallianpur_1995,Lions_2000}, commonly involving the decomposition of system parameters into deterministic components and stochastic components influenced by noise, such as white or color noise \cite{DaPrato_1992,Alos_2002}, to extract valuable statistical information. In this work, we focus on the following steady-state incompressible  SNS system
\begin{eqnarray}\label{SNS}
	\begin{cases}
		-\nu\Delta \textbf{u}+(\textbf{u}\cdot\nabla) \textbf{u}+\nabla p=F+\sigma\frac{d\textbf W}{d\textbf x} \quad \text{ in } \Omega, \\
		\nabla\cdot\textbf{u}=0\quad \text{in }\Omega ,\\
		\textbf{u}|_{\partial\Omega}=\textbf{0},
	\end{cases}
\end{eqnarray}
where $\Omega\subset \mathbb{R}^2$ is connected and Lipschitz bounded, $\textbf{u}=(u_1,u_2)^{T}$ denotes the velocity field, $p$ the pressure, $\nu$ the given constant kinematic viscosity, and $F+\sigma\frac{d\textbf W}{d\textbf x}$ is the given body force with an additive spatial random noise, a modeling approach widely employed in various fields such as fluid dynamics, random simulation, and a broad range of related fields.

The system described by 
\eqref{SNS} plays a pivotal role in numerous engineering applications. However, traditional numerical techniques often struggle to effectively balance computational time and error, particularly when dealing with the complex dynamics inherent in such systems. To address these challenges, the splitting method, also known as operator splitting, has emerged as a prevalent numerical approach for solving the NS  flow  \cite{Gunzburger_2016,Deugoue_2021,Bessaih_2019,Zhao_2017,Deugoue_2014,Dorsek_2012}. This approach entails decomposing the governing equations into a series of simpler sub-problems, which can then be solved sequentially or iteratively.  By doing so, the splitting method significantly enhances computational efficiency and accuracy, making it particularly effective for handling the intricate nonlinear terms and constraints present in the Navier-Stokes equations.  In order to apply the splitting method, the Navier-Stokes equations are partitioned into sub-problems based on the convection and diffusion terms\cite{Zhao_2020}.  A prevalent approach entails dividing the equations into a convection step and a diffusion step, each of which can be solved using specialized numerical techniques tailored to the specific characteristics of the sub-problem.	During the convection step, the convection term is treated explicitly, while the diffusion term is either neglected or approximated.  This stage involves solving the equations without considering diffusion effects, yielding an intermediate velocity field.  A variety of numerical methods, such as the finite volume method \cite{Capdeville_2022} or finite difference method \cite{Lucca_2023}, can be employed to efficiently solve the convection step.	Subsequently, in the diffusion step, the diffusion term is treated implicitly, whereas the convection term is either neglected or approximated.  This stage encompasses solving the equations considering diffusion effects, typically utilizing implicit numerical schemes.  This can be accomplished by solving a linear system of equations, employing methods like the implicit Euler method \cite{Cherfils_2016} or implicit Crank-Nicolson method \cite{Cho_2023} . By sequentially addressing the convection and diffusion steps, the splitting method effectively decouples the effects of convection and diffusion, facilitating a more straightforward and efficient numerical solution of the stochastic NS equations. However, it is crucial to recognize that the splitting method introduces errors due to the neglect or approximation of terms in each step.  The accuracy and stability of the method are dependent upon the specific splitting scheme chosen and the numerical methods employed in each step.

The splitting method has predominantly been employed for addressing complex high-dimensional equation problems. Early research on this method can be traced back to the review papers by Strang \cite{Strang_1968} and Marchuk \cite{Marchuk_1968}, who were among the first pioneers in this field.  Additionally, the splitting method has been successfully applied to analyze error estimation and convergence results for a wide range of equations, such as the fourth-order elliptic equation, Buregers' equation \cite{Seydaoglu_2016}, the two-dimensional and three-dimensional Allen-Cahn equation \cite{Brehier_2016}, and others \cite{Gunzburger_2013}. Commonly used splitting techniques include the projection method, Yoshida splitting method, and others. In the 1960s, Chorin  \cite{Chorin_1968} and Team \cite{Temam_1969} proposed the projection method , which aims to solve equations by splitting the nonlinearity from the incompressibility constraint.  This approach has since inspired numerous variants and extensions \cite{Guermond_2006,Brown_2001,Prohl_2009,Olshanskii_2010}. Initially, the equations are solved without taking the incompressibility constraint into account, yielding an intermediate velocity field.  In the subsequent step, this intermediate velocity field is projected onto the space of solenoidal (divergence-free) vector fields. By segregating the nonlinearity from the incompressibility constraint and solving the equations in two distinct steps, the projection method offers computational advantages and improves memory usage. The Yoshida splitting method shares the similar spirit with the projection method but differs in the decomposition process. In the Yoshida splitting method, the system is fully decomposed after the matrix is constructed, before the system is split.  This method proves to be particularly effective for problems in which the time scale is smaller than the spatial scale and requires adaptive time steps. Consequently, it  offers adaptive time step capabilities and overcomes the challenges associated with implementing artificial boundary conditions
\cite{Saleri_2006,Badia_2008,Bertagna_2016}.

In order to enhance the stability and computational efficiency of the solutions to equation \eqref{SNS}, we  propose an advanced splitting scheme tailored for the steady-state incompressible  NS flow.

Setting $\textbf{u}=\bm{\xi}+\bm{\eta}$, we have that $(\bm{\xi},p_1)$ is the solution of the deterministic NS equation
\begin{eqnarray}\label{SSNS1}
	\begin{cases}
		-\nu\Delta \bm{\xi}+(\bm{\xi}\cdot\nabla) \bm{\xi}+\nabla p_1=F \quad \text{ in } \Omega,\\
		\nabla\cdot\bm{\xi}=0\quad \text{in }\Omega ,\\
		\bm{\xi}|_{\partial\Omega}=\textbf{0}.
	\end{cases}
\end{eqnarray}
and $(\bm{\eta},p_2)$ satisfies the following problem
\begin{eqnarray}\label{SSNS2}
	\begin{cases}
		-\nu\Delta \bm{\eta}+(\bm{\eta}\cdot\nabla)\bm{\eta}+(\bm{\eta}\cdot\nabla) \bm{\xi}+(\bm{\xi}\cdot\nabla)\bm{\eta}
		+\nabla p_2=\sigma\frac{d\textbf W}{d\textbf x}\quad \text{ in } \Omega, \\
		\nabla\cdot\bm{\eta}=0\quad \text{in }\Omega ,\\
		\bm{\eta}|_{\partial\Omega}=\textbf{0}.
	\end{cases}
\end{eqnarray}
The linearized Navier-Stokes equations (LNSE) is derived from the Navier-Stokes equations under the assumption of small disturbances in fluid flow. This simplification reduces the complexity of the inherently nonlinear equations, facilitating a more streamlined analysis of fluid dynamics problems, particularly applicable when fluid motion is relatively slow or disturbances are small. In the late 19th century, British scientist Osborne Reynolds experimentally observed  that fluid flow transitions from laminar to turbulent states as the Reynolds number exceeds a certain threshold  \cite{Reynolds_1883}.To delve deeper into understanding this transition between laminar and turbulent flows, researchers attempted to linearize the vorticity transport equation (Navier-Stokes equations). In 1895, British mathematician Horace Lamb systematically derived the linearized form of the Navier-Stokes equations, laying the theoretical foundation for LNSE  \cite{Lamb_1895}. Since then, LNSE has been widely employed in fluid dynamics research, particularly in studies of fluid stability, boundary layer flows, Reynolds number effects, laminar-turbulent transition, intermittent flows, and fluid wave analysis. With advancements in computing technology, LNSE has become a crucial tool for numerical simulations in fluid mechanics, applied across various fields including aerospace, ocean engineering, and environmental science for tasks such as aircraft wing design, ocean fluid dynamics analysis, and hydrological modeling \cite{Farrell_1993,Jovanovic_2001,Barbu_2011,Grenier_2022}. Building upon the principles of LNSE, which consider small perturbations in fluid motion, we propose a modified splitting method to simplify the analysis by focusing on the linearized stochastic equations, omitting the nonlinear parts $(\bm{\eta}\cdot\nabla)\bm{\eta}$ of the stochastic equation \eqref{SSNS2} , aiming to provide a more concise and manageable framework for analytical and numerical investigations.

The modified splitting method satisfies equation \eqref{SSNS1} and the linear stochastic equation below
\begin{eqnarray}\label{SSNS3}
	\begin{cases}
		-\nu\Delta \tilde{\bm{\eta}}+(\tilde{\bm{\eta}}\cdot\nabla) \bm{\xi}+(\bm{\xi}\cdot\nabla)\tilde{\bm{\eta}}
		+\nabla \tilde p_2=\sigma\frac{d\textbf W}{d\textbf x}\quad \text{ in } \Omega, \\
		\nabla\cdot\tilde{\bm{\eta}}=0\quad \text{in }\Omega ,\\
		\tilde{\bm{\eta}}|_{\partial\Omega}=\textbf{0}.
	\end{cases}
\end{eqnarray}

This paper aims to propose the splitting method and a modified splitting method  to simplify the analysis process. The primary advantage of these methods lies in their ability to relaxe the conditions of existence, stability, and uniqueness of the solution, while ensuring that the efficiency and accuracy of Monte Carlo simulation will be greatly improved under the condition that the error is relatively small.

The paper is organized as follows: in Section \ref{sec2}, we introduce the concepts and properties of functions and spatial symbols and give the weak schemes for equation\eqref{SNS} and splitting scheme \eqref{SSNS1}-\eqref{SSNS2}. Furthermore, we delve into an extensive analysis of the properties exhibited by solutions to steady-state SNS equation. In Section \ref{sec3}, we first establish the equivalence between the splitting method and the original equation and identify the sources of statistical error. Then, we prove the stability, existence, and uniqueness of the solutions obtained through the splitting method. Additionally, we consider a modified splitting format (equation \eqref{SSNS1} and equation \eqref{SSNS3}) that disregards the nonlinear terms and analyze the properties of the solution. In Section \ref{sec4}, we rigorously explore the Galerkin finite element discrete scheme of the splitting method and the modified splitting format, while also give the stability, existence and uniqueness conditions of the finite element solution. Moreover, a particular emphasis is placed on analysing the discrete error associated with the splitting scheme.  In numerical experiments, the splitting method, the modified splitting method and the finite element solution were compared under different sample sizes and random noise, and the advantages of the splitting method in stability conditions and computational efficiency were demonstrated, as described in Section \ref{sec5}. We address some conclusions and discussions in Section \ref{sec6}.

\section{Preliminariesg}\label{sec2}
In this section, we provide an exposition of the standard functions, space symbols, and mathematical notations that will be employed consistently across this paper. Subsequently, we undertake a comprehensive review of the definition of the variational weak solution pertaining to problem \eqref{SNS}. Furthermore, we rigorously establish the correlation properties inherent in the weak solutions for SNS equation. The outcomes of these auxiliary investigations assume utmost significance in comprehending the correlation properties exhibited by the solutions acquired through the employment of the splitting method, as well as in conducting error analyses in the subsequent sections.
\subsection{Notation}
Due to $\Omega$ denote an open, bounded, possibly multiply connected spatial domain in $\mathbb{R}^2$, we denote by $L^2(\Omega)$ the space of all square integrable functions with inner product and norm ($\cdot$, $\cdot$) and $\|\cdot\|$ , respectively 
$$
(q,v)=\int_\Omega qv~d\Omega \quad  \text{ and~~~} \|v\|= (v,v)^\frac{1}{2}.
$$
Define the constrained space
$$
L^2_0(\Omega)= \left\{q\in L^2(\Omega):\int_\Omega q d\Omega= 0 \right\}.
$$
For given positive integers $k\in \mathbb{N}$, we define the standard Sobolev space \cite{Adams_1978}
$$
H^k(\Omega)= \left\{q\in L^2(\Omega):D^s q\in L^2(\Omega) \quad \text{for~} s=1, \ldots, k \right\},
$$
with inner product $(\cdot, \cdot)_k$  and norm $\|\cdot\|_k$
$$
\|\cdot\|_k=\|\cdot\|_{H^k(\Omega)}=\left(\left\| q\right\|^2 +\sum\limits_{s=1}^k \left\| D^s q\right\|^2 \right )^{\frac{1}{2}},
$$
where $D^s$ denotes any and all derivatives of order $s$.

Clearly, $H^0(\Omega)= L^2(\Omega)$. Moreover, define 
$
H^k_0(\Omega)=\left\{q\in H^k(\Omega): q=0 \text{ on } \partial \Omega \right\}
$
is the closure of  $C^{\infty}_0(\Omega)$ under the norm $\|\cdot\|_k$. Specially, the space $H^1_0(\Omega)$ has the associated semi-norm \cite{Gunzburger_1989}
$$
|\cdot|_1:=\left(\sum\limits_{i=1}^n \left\| \frac{\partial q}{\partial x_i}\right\|^2 \right )^{\frac{1}{2}}.
$$

Besides, we define the dual space of $H^k_0(\Omega)$
$$
H^{-k}(\Omega):=(H^k_0(\Omega))^* \quad \text{ with~ ~ } \|q\|_{H^{-k}(\Omega)} = \sup\limits_{v\in H^{k}_0}\frac{(q,v)}{\|v\|_{H^k(\Omega)}}.
$$
To describe the space of random functions, we consider the probability space $\mathcal{P}:=(\Theta,\mathcal{F},\mu )$, where $\Theta$ is the sample space, $\mathcal{F}$ denotes the $\sigma$-algebra of evants and $\mu$ denotes the probability measure,  with the inner product and norm defined as
$$
(f_1,f_2)_\mu:= \int_\Theta f_1 f_2~d \mu \quad  \text{ and~~~} \|f_1\|_{L^2_\mu(\Theta)}= (f_1,f_1)_{\mu}^\frac{1}{2},
$$
respectively.  We can then define the expectation operator $\mathbb E$ as
\begin{equation*}
	\mathbb E[f]= \int_\Theta f ~d \mu.
\end{equation*}
Furthermore, we define the stochastic tensor space $L^2(\Theta; H^k(\Omega))$,  which is isomorphic to $L^2(\Theta)\otimes H^k(\Omega)$ with a linear map from $L^2(\Theta)\otimes H^k(\Omega)$ to $L^2(\Theta; H^k(\Omega))$ \cite{Smolyak_1963,Babuska_2004}, namely, 
$$L^2(\Theta; H^k(\Omega))\simeq L^2(\Theta)\otimes H^k(\Omega).$$
Define the stochastic Sobolev space 
$$L^2\left(\Theta;H^k(\Omega)\right)=\left\{ w:\Theta\times \Omega\to \mathbb R\lvert \quad \mathbb E\left[\|w\|^2_{H^k
	(\Omega)}\right]<\infty \right\}, $$
with the norm
$$ \|w\|_{L^2\left(\Theta; H^k(\Omega)\right)}:=\left(\mathbb E\left[\|w\|^2_{H^k(\Omega)}\right]\right)^\frac{1}{2}=\left(\int_{\Theta} \|w\|^2_{H^k(\Omega)}~d \mu\right)^\frac{1}{2}.$$

For simplicity, we define the following special spatial symbols
$$
\mathbb V :=L^2\left(\Theta; H^1(\Omega)\right),\quad \mathbb V_0 :=L^2\left(\Theta; H^1_0(\Omega)\right),$$
$$ \mathbb W :=L^2\left(\Theta; L^2(\Omega)\right),\quad\mathbb W_0 :=L^2\left(\Theta; L^2_0(\Omega)\right),
$$
$$
\textbf{V}_{div} :=\left\{\textbf{v}\in H^1_0(\Omega): \nabla \cdot \textbf{v}=0  \quad\text{in } \Omega  \right\},\quad
\textbf W_{div} :=\left\{\textbf{v}\in L^2_0(\Omega): \nabla \cdot \textbf{v}=0  \quad\text{in } \Omega  \right\},$$
$$
\mathbb V_{div} :=\left\{\textbf{v}\in \mathbb V_0: \nabla \cdot \textbf{v}=0  \quad\text{in } \Omega  \right\},\quad\mathbb W_{div} :=\left\{\textbf{v}\in \mathbb W_0: \nabla \cdot \textbf{v}=0  \quad\text{in } \Omega  \right\},
$$
where $\|\cdot\|_{\mathbb V}=\left(\mathbb E\left[\|\cdot\|^2_1\right]\right)^\frac{1}{2},$  $\|\cdot\|_{\mathbb W}=\left(\mathbb E\left[\|\cdot\|^2\right]\right)^\frac{1}{2}$ and  $\|\cdot\|_{\mathbb V_{0}}=\left(\mathbb E\left[|\cdot|^2_1\right]\right)^\frac{1}{2}.$  

In order to give the meaning of the stochastic term $\frac{d\textbf W}{d \textbf x}$ , we need to review the spatial noise approximation technique \cite{Gyongy_1997,Ming_2013,Gunzburger_2019}, which is used in signal processing, image processing, and stochastic simulation, among others. In this study, we focus on the grid-based approach\cite{Nobile_2008}, which involves partitioning the space into a regular grid and generating discretized white noise samples at each grid point. This discretization procedure introduces some regularity while preserving the characteristics of independence and Gaussian distribution. We begin by subdividing the spatial domain $\Omega$ into $N$ non-overlapping subdomains $\{\Omega_n\}^N_{n=1}$. Considering the $\textbf{x}=[x,y]^T$ is a two-dimensional vector, regional $\Omega=[0,1]\times[0,1].$ Let $\{x_i^{(n)}=nh\}_{n=0}^N$ be a partition of $[0,1]$ with $h=\frac{1}{n}, i=1,2.$ The random field $\textbf W$ represents a discretized version of a stochastic process defined over the spatial domain. A  piecewise constant approximation with respect to the subdivision $\{\Omega_n\}^N_{n=1}$ of the random field is given by\cite{Allen_1998,Du_2002,Yan_2004}
\begin{equation}\label{noise}
	\frac{d\textbf W}{d \textbf x}= \sum_{k=1}^N \frac{\sigma}{\sqrt{V_k}}\zeta_k(\omega)\chi_k(\textbf x), 
\end{equation}
where the coefficients $\sigma$ is variance, $V_k$ denotes the volume of the subdomain  $\Omega_n$, $ \{\chi_k(\textbf x)\}$ denotes the characteristic (or indicator) function corresponding to the subdomain , and $\zeta_k(\omega)$ denotes  an $\mathcal{N}(0,1)$ i.i.d. random number,
\begin{align*}
	\begin{split} 
		V_k=h^2=\frac{1}{n^2},\quad \quad \chi_k(\textbf x)= \left \{ 
		\begin{array}{ll}
			1,           &x_k\leq x \leq x_{k+1}, y_k\leq y \leq y_{k+1},\\
			0,     &otherwise.
		\end{array} 
		\right. 
	\end{split}
\end{align*}

Specifically, consider $\sigma=1$, the discrete discretized white noise form for the piecewise constant approximation is given by
$$
\lim_{N\to \infty}\mathbb E\left(\frac{d\textbf W(\textbf{x})}{d \textbf x}\cdot \frac{d\textbf W(\textbf{x}')}{d \textbf x} \right)=\delta (\textbf{x}-\textbf{x}'),
$$
where $\delta$ denotes the usual Dirac $\delta$-function and $\mathbb E$ the expectation. Via discretization, white noise has been reduced to the case of a large but finite number of parameters. 

\subsection{Weak formulation}
In order to represent the Galerkin-type weak scheme of the stochastic steady-state NS problem, we first  define some continuous formulations. 

The continuous bilinear forms
\begin{gather}
	a(\textbf{u},\textbf{v}):= \nu\int_\Omega\nabla\textbf{u}:\nabla\textbf{v}d\Omega, \quad \forall \textbf u,\textbf v\in H^1(\Omega),\nonumber\\
	b(\textbf{u}, q):= -\int_\Omega q\nabla\cdot\textbf{u}d\Omega, \quad \forall  q\in L^2(\Omega), \textbf u\in H^1(\Omega),\nonumber
\end{gather}
where $\nabla\textbf{u}:\nabla\textbf{v}=\sum\limits_{i,j=1}^{n}\frac{\partial u_i}{\partial x_j}\frac{\partial v_i}{\partial x_j}$.

The continuous trilinear form
\begin{gather}
	c(\textbf u,\textbf w,\textbf v):= \int_\Omega \textbf{u}\cdot\nabla\textbf{w}\cdot\textbf{v}d\textbf{x},  \quad \forall \textbf u, \textbf v, \textbf w\in H^1(\Omega),\nonumber
\end{gather}
where $\textbf{u}\cdot\nabla\textbf{v}\cdot\textbf{w}=\sum\limits_{i,j=1}^{n} u_j\frac{\partial v_i}{\partial x_j}w_i$.

According to some basic inequalities, we have the following properties
\begin{flalign}
	|a(\textbf u,\textbf v)|&\le C_a\|\textbf u\|_1\|\textbf
	v\|_1, \qquad \forall \textbf u, \textbf v \in H^1(\Omega),\label{a_1}\\
	(\textbf u,\textbf v)&\le \|\textbf u\|_{H^{-1}(\Omega)}|\textbf
	v|_1, \qquad \forall \textbf u\in H^{-1}(\Omega),\textbf v\in H^1(\Omega),\label{dual}\\
	|b(\textbf{v}, q)|&\le C_b\|\textbf v\|_1\|q\|, \qquad\quad \forall \textbf v \in H^1(\Omega),q \in L^2(\Omega),\label{b_1}\\
	|c(\textbf u,\textbf w,\textbf v)|&\le
	C_c\|\textbf u\|_1\|\textbf w\|_1\|\textbf v\|_1,\qquad \forall \textbf u, \textbf v, \textbf w\in H^1(\Omega),\label{c_2}\\
	\mathbb E[c(\textbf u,\textbf w,\textbf v)] &\le
	C_e\|\textbf u^T \textbf w\|_{\mathbb W}\|\textbf v\|_{\mathbb V},\qquad \forall \textbf u, \textbf v, \textbf w\in \mathbb V.\label{c_5}
\end{flalign}
Moreover, we have the properties \cite{Temam_2001}
\begin{gather}
	c(\textbf u,\textbf w,\textbf v)=
	-c(\textbf u,\textbf v,\textbf w),\quad \forall \textbf v,\textbf w \in {H^1_0(\Omega)}, \label{c_3}\\
	c(\textbf u,\textbf v,\textbf v)= 0,\quad \forall \textbf v,\textbf w \in {H^1_0(\Omega)}. \label{c_4}
\end{gather}

The weak form of  the homogeneous stochastic NS problem\eqref{SNS} is as follows: find $(\textbf u,p)\in(\mathbb V_0, \mathbb W_0)$ such that
\begin{equation}\label{w_SNS1}
	\begin{split}
		\mathbb E[a(\textbf u,\textbf w)]+\mathbb E[c(\textbf u,\textbf u,\textbf w)] +\mathbb E[b(\textbf{w}, p)] &= \mathbb E\left[\left(\textbf F+\sigma\frac{d\textbf W}{d\textbf x}, \textbf{w}\right)\right],\\
		\mathbb E[b(\textbf{u}, q)]&= 0,
	\end{split}
\end{equation}
for all test function $(\textbf{w},q)\in(\mathbb V_0, \mathbb W_0)$.

\begin{remark}[Variations of weak format]
	The problem \eqref{w_SNS1} can be formulated equivalently in the following way:  find $(\textbf u,p)\in(\mathbb V_0, \mathbb W_0)$ such that
	\begin{equation}\label{w_SNS2}
		\mathbb E[a(\textbf u,\textbf w)]+\mathbb E[c(\textbf u,\textbf u,\textbf w)] +\mathbb E[b(\textbf{w}, p)]-\mathbb E[b(\textbf{u}, q)]= \mathbb E\left[\left(\textbf F+\sigma\frac{d\textbf W}{d\textbf x}, \textbf{w}\right)\right],
	\end{equation}
	for all test function $(\textbf{w},q)\in(\mathbb V_0, \mathbb W_0).$
	
	Notice that the following weak formulation which is equivalent to \eqref{w_SNS1}: find $\textbf u
	\in \mathbb V_{div}$ such that
	\begin{equation}\label{w_SNS3}
		\begin{split}
			a(\textbf u,\textbf w)+c(\textbf u,\textbf u,\textbf w) = \left(\textbf{F}+\sigma\frac{d\textbf W}{d\textbf x}, \textbf{w}\right), a.s., \quad \forall \textbf w\in \mathbb V_{div},
		\end{split}
	\end{equation}
	where  a.s. refers to almost surely in the sense of the associated probability measure. 
\end{remark}
Obviously, the bilinear form which couples velocity and pressure is the same as  for Stokes and the spaces are the same, too. Therefore, given a solution $\textbf{u}$ of the \eqref{w_SNS3}, there exists a unique pressure $p\in \mathbb W_0$  such that $(\textbf{u}, p)$ solves  \eqref{w_SNS1}.

\subsection{The existence and uniqueness theorem  }
In this part, we prove the stability, existence and uniqueness of the solution of the Stochastic Steady-State NS equation \eqref{w_SNS1}. 
\begin{lemma} [inf-sup condition] \label{inf-sup condition}
	There exits a constant $\beta > 0$ such that \cite{Girault_1986,Douglas_1988}
	\begin{equation}\label{lem1}
		\sup_{0\neq\textbf w\in \mathbb V_0}\frac{\mathbb E[b(\textbf{w}, q)]}{\|\textbf w\|_{\mathbb V_0}} \geq
		\beta \|q\|_{\mathbb W}.
	\end{equation}
\end{lemma}
\begin{lemma}[Stability of the solution] \label{lem2} 
	Let $(\textbf u,p)\in(\mathbb V_0, \mathbb W_0)$ be any solution of  \eqref{w_SNS1}, then 
	\begin{gather}
		\|\textbf u\|_{\mathbb V_0}\le \frac{1}{\nu}\mathbb E\left[\left \|\textbf{F}+\sigma\frac{d\textbf W}{d\textbf x}
		\right\|_{H^{-1} (\Omega)}\right], \label{l1}\\
		\| p\|_{\mathbb W_0}\le \frac{2}{\beta}\mathbb E\left[ \left\| \textbf{F}+ \sigma\frac{d\textbf W}{d\textbf x}
		\right\|_{H^{-1}(\Omega)}\right]+  \frac{C_{c_1}}{\beta \nu^2}\mathbb E\left[\left\|\textbf{F}+\sigma\frac{d\textbf W}{d\textbf x}
		\right\|^2_{H^{-1} (\Omega)}\right],\label{l2}
	\end{gather}
	
	where the constant $C_{c_1} > 0$ such that $|c(\textbf u,\textbf u,\textbf w)|\le
	C_{c_1}|\textbf u|_1^2|\textbf w|_1.$
\end{lemma}	
\begin{proof}  we choose test function $(\textbf{w}, q)= (\textbf{u}, p)$ in \eqref{w_SNS2}
	\begin{equation}\label{16}
		\mathbb E\left[a(\textbf u,\textbf u)\right]+\mathbb E\left[c(\textbf u,\textbf u,\textbf u)\right] =\mathbb E\left[\left(\textbf{F} +\sigma\frac{d\textbf W}{d\textbf x}, \textbf{u}\right)\right], \nonumber
	\end{equation}
	
	using \eqref{c_4}, we can write
	\begin{equation}\label{17}
		\nu \|\textbf u\|^2_{\mathbb V_0}= \mathbb E\left[a(\textbf u,\textbf u)\right] = \mathbb E\left[\left(\textbf{F}+\sigma\frac{d\textbf W}{d\textbf x},\textbf{u} \right)\right]. \nonumber
	\end{equation}
	Moreover, by \eqref{dual}, we obtain the following relation
	\begin{equation}\label{18}
		\nu \|\textbf u\|^2_{\mathbb V_0}\le \mathbb E\left[\left \|\textbf{F}+\sigma\frac{d\textbf W}{d\textbf x}
		\right\|_{H^{-1} (\Omega)}\right]\|\textbf u\|_{\mathbb V_0}.\nonumber
	\end{equation}
	Thus, inequality \eqref{l1} holds. According to the inf-sup condition and \eqref{dual}, one obtains for the pressure 
	\begin{equation}
		\begin{split}
			\| p\|_{\mathbb W_0}&\le  \frac{1}{\beta}\sup_{0\neq\textbf w\in \mathbb V_0}\frac{\mathbb E[b(\textbf{w}, p)]}{\|\textbf w\|_{\mathbb V_0}} \\
			&=\frac{1}{\beta}\sup_{0\neq\textbf w\in \mathbb V_0}\frac{\mathbb E\left[(\textbf{F}+\sigma\frac{d\textbf W}{d \textbf{x}}, \textbf{w})\right]-\mathbb E[a(\textbf u,\textbf w)]-\mathbb E[c(\textbf u, \textbf u,\textbf w)]}{\|\textbf w\|_{\mathbb V_0}} \nonumber,
		\end{split}
	\end{equation}
	namely,
	\begin{equation}
		\begin{split}
			\| p\|_{\mathbb W_0}\le  \frac{1}{\beta} \left(\mathbb E\left[\left \|\textbf{F}+\sigma\frac{d\textbf W}{d\textbf x}
			\right\|_{ H^{-1} (\Omega)} \right]+\nu\|\textbf u\|_{\mathbb V_0}+C_{c_1}\|\textbf u\|^2_{\mathbb V_0}\right).\nonumber
		\end{split}
	\end{equation}
	Substitute \eqref{l1} into the inequality above and it is evident that inequality \eqref{l2} holds.
\end{proof}

The next theorem estabilishes the existence and uniqueness of the solution.

\begin{theorem}[The existence and uniqueness theorem of the solution for small data] \label{thm1}
	
	Let $\Omega \subset \mathbb R^2$ be a bounded domain with Lipschitz boundary $\partial \Omega $ and the random body force $\textbf{F}+\sigma\frac{d\textbf W}{d\textbf x} \in L^2 \left(\Theta; H^{-1}(\Omega)\right), $ if
	\begin{equation}\label{th0}
		\begin{split}
			\left(\|\textbf{F}\|^2_{L^2(\Omega)}+\left\|\sigma\frac{d\textbf W}{d\textbf x}\right\|^2_{\mathbb W}\right)^\frac{1}{2}\textless \frac{\nu ^2}{\Hat{C}},
		\end{split}
	\end{equation}
	where  $\Hat{C} = C_{c_2}C_\Omega,$  $C_{c_2} > 0$ such that $c(\textbf{v}_1-\textbf{v}_2,\textbf{u}_1, \textbf{u}_1-\textbf{u}_2)\le
	C_{c_2}|\textbf{v}_1-\textbf{v}_2|_1|\textbf{u}_1|_1|\textbf{u}_1-\textbf{u}_2|_1$  and  the constant $C_\Omega > 0$ such that $\|\textbf{u}\| \le C_\Omega|\textbf{u}|_1$, then there exists a unique pair of stochastic processes $(\textbf u, p)\in (\mathbb V_0, \mathbb W_0)$  satisfying equations\eqref{w_SNS1} a.s..
\end{theorem}

\begin{proof}Since equation \eqref{w_SNS3} is equivalent to \eqref{w_SNS1}, so the equivalent problem \eqref{w_SNS3} in the divergence-free subspace $\mathbb V_{div}$ is considered. 
	
	Define the bilinear form
	\begin{equation}\label{26}
		\begin{split}
			A_v(\textbf u,\textbf w) := \mathbb E[a(\textbf u,\textbf w)+c(\textbf{v},\textbf u,\textbf w)] ,\nonumber
		\end{split}
	\end{equation}
	for given $\textbf{v}\in \mathbb V_{div}$. 
	
	Clearly, we  obtain
	\begin{equation}\label{th1}
		|A_v(\textbf u,\textbf w)| \le C_v(\|\textbf u\|^2_{\mathbb V_0})^\frac{1}{2}(\|\textbf w\|^2_{\mathbb V_0})^\frac{1}{2} , 
	\end{equation}
	where $C_v$ is a constant related to $\textbf{v}.$
	$$
	A_v(\textbf u,\textbf u) =\mathbb E[a(\textbf u,\textbf u)] = \nu \|\textbf u\|^2_{\mathbb V_0} ,
	$$
	so that $A_v(\textbf u,\textbf w)$ is coercive. Moreover, 
	\begin{equation}\label{F}
		\left\|\textbf{F}+\sigma \frac{d\textbf W}{d\textbf x}\right\|^2_{\mathbb W} \le 2\left(\|\textbf{F}\|^2+\left\|\sigma \frac{d\textbf W}{d\textbf x}\right\|^2_{\mathbb W}\right),
	\end{equation}
	so that $\mathbb E\left(\textbf{F}+\sigma\frac{d\textbf W}{d\textbf x}, \cdot\right)$ is a continuous linear functional on $\mathbb V_{div}$. Therefore, by Lax-Milgram theorem, for any given $\textbf{v}\in \mathbb V_{div},$ there exists a unique stochastic processes $\textbf{u}\in \mathbb V_{div}$  such that
	\begin{equation}\label{th2}
		\begin{split}
			A_v(\textbf u,\textbf w) = \mathbb E\left(\textbf{F}+\sigma\frac{d\textbf W}{d\textbf x}, \textbf{w}\right), \quad \forall \textbf w\in \mathbb V_{div} ,
		\end{split}
	\end{equation}
	that is to say
	\begin{equation}
		\begin{split}
			a(\textbf u,\textbf w)+c(\textbf{v},\textbf u,\textbf w) = \left(\textbf{F}+\sigma\frac{d\textbf W}{d\textbf x}, \textbf{w}\right), a.s., \quad \forall \textbf w\in \mathbb V_{div}.\nonumber
		\end{split}
	\end{equation}
	Define the nonlinear operator $\pi: \textbf{v} \longrightarrow \textbf{u}$ through \eqref{th2}, it will be shown that $\pi$ is a contraction mapping. For any $\textbf{v}_1, \textbf{v}_2\in \mathbb V_{div}$, there exists a $\gamma \in(0,1) $ such that 
	\begin{equation}
		\begin{split}
			\left\|\pi(\textbf{v}_1)-\pi(\textbf{v}_2) \right\|_{\mathbb V_0} = \left\|\textbf{u}_1-\textbf{u}_2\right\|_{\mathbb V_0} \le \gamma \left\|\textbf{v}_1-\textbf{v}_2\right\|_{\mathbb V_0} , \nonumber
		\end{split}
	\end{equation}
	where $\pi(\textbf{v}_1) = \textbf{u}_1, \pi(\textbf{v}_2) = \textbf{u}_2$, so there exists a unique stochastic process $\textbf{v}\in \mathbb V_{div}$  such that $\textbf{v} = \textbf{u}  \quad \mu-a.e..$ 
	
	Indeed, taking $\textbf{u}_1,\textbf{u}_2$ to satisfy the equation \eqref{th2}, we obtain
	\begin{equation}
		\mathbb E[a(\textbf{u}_1-\textbf{u}_2, \textbf w )]+\mathbb E[c(\textbf{v}_1-\textbf{v}_2,
		\textbf{u}_1,\textbf w)]+\mathbb E[c(\textbf{v}_2,
		\textbf{u}_1-\textbf{u}_2,\textbf w)]= 0.\nonumber
	\end{equation}
	Let $\textbf{w}=\textbf{u}_1-\textbf{u}_2$, the last term on the left-hand side vanishes because of \eqref{c_4} and one obtains with \eqref{c_2} and H\"older inequality
	\begin{align}
		\nu \left\|\textbf{u}_1-\textbf{u}_2\right\|^2_{\mathbb V_0} &\le \left | \mathbb E[c(\textbf{v}_1-\textbf{v}_2,
		\textbf{u}_1,\textbf{u}_1-\textbf{u}_2)]\right|  \nonumber\\
		&\le C_{c_2}\mathbb E[|\textbf{v}_1-\textbf{v}_2|_1|\textbf{u}_1|_1|\textbf{u}_1-\textbf{u}_2|_1] \nonumber \\
		&\le C_{c_2}\left(\left\|\textbf{v}_1-\textbf{v}_2\right\|^2_{\mathbb V_0}\left\|\textbf{u}_1\right\|^2_{\mathbb V_0}\right)^\frac{1}{2}\left\|\textbf{u}_1-\textbf{u}_2\right\|_{\mathbb V_0}.  \nonumber
	\end{align}
	Thus it follows from inqualities \eqref{l1} and  \eqref{F}
	\begin{equation}
		\begin{split}
			\left\|\textbf{u}_1-\textbf{u}_2\right\|_{\mathbb V_0} 
			\le \frac{C_{c_2}C_\Omega}{\nu ^2}\left(\|\textbf{F}\|^2+\left\|\sigma \frac{d\textbf W}{d\textbf x}\right\|^2_{\mathbb W}\right)^\frac{1}{2}\left\|\textbf{v}_1-\textbf{v}_2\right\|_{\mathbb V_0}.
		\end{split}
	\end{equation}
	Naturally, if the following conditions are met, then $\pi$ is a contraction mapping
	\begin{equation}
		\frac{C_{c_2}C_\Omega}{\nu ^2}\left(\|\textbf{F}
		\|^2+\left\|\sigma \frac{d\textbf W}{d\textbf x}\right\|^2_{\mathbb W}\right)^\frac{1}{2}\textless 1.
	\end{equation}
	Obviously, inequality\eqref{th0} holds. According to the principle of contraction mapping, the existence and uniqueness of the solution are proved.
\end{proof}

\section{Splitting method for  the Stochastic flow}\label{sec3}

In this section, we develop and analyze the splitting approach for resolving the  steady-state SNS system. The weak Galerkin formulation of the splitting method has been preliminarily studied. Following this, a thorough exposition of the benefits associated with the splitting method is presented, showcasing specific focus on the origins of theoretical errors and properties of the solutions, such as stability, existence, and uniqueness conditions. Additionally, we ponder on the scenario of disregarding the nonlinear term in the splitting method under particular conditions and suggest a revised format for splitting. We then proceed to scrutinize the properties of the solution generated.

Similar to Section $\bm{2.2}$, the weak formulations of  equations\eqref{SSNS1} and\eqref{SSNS2} are presented as follows: find $\bm{\eta}\in \mathbb V_0,\bm{\xi}\in H^1_0(\Omega) , p_1,p_2\in\mathbb W_0$ such that
\begin{equation}\label{w_sSNS1}
	\begin{split}
		\mathbb E[a(\bm{\xi},\textbf v)]+\mathbb E[c(\bm{\xi},\bm{\xi},\textbf v)] +\mathbb E[b(\textbf{v}, p_1)] &= \mathbb E\left[\left(\textbf F, \textbf{v}\right)\right],\\
		\mathbb E[b(\bm{\xi}, q)]&= 0,
	\end{split}
\end{equation}
and
\begin{equation}\label{w_sSNS2}
	\begin{split}
		\mathbb E[a(\bm{\eta},\textbf{v})]+\mathbb E[c(\bm{\eta}, \bm{\eta}, \textbf{v})] +\mathbb E[c(\bm{\eta},\bm{\xi},\textbf v)]+\mathbb E[c(\bm{\xi}, \bm{\eta},\textbf{v})]+&\mathbb E[b(\textbf{v}, p_2)] = \mathbb E\left[\left(\sigma\frac{d\textbf W}{d\textbf x}, \textbf{v}\right)\right],\\
		&\mathbb E[b(\bm{\eta}, q)]= 0,
	\end{split}
\end{equation}

for all test function $(\textbf{v},q)\in(\mathbb V_0, \mathbb W_0)$.

\begin{remark}[Variations of weak format]
	
	The problem \eqref{w_sSNS1}-\eqref{w_sSNS2} can be equivalently expressed in the following manner
	$$
	a(\bm{\xi},\textbf v)+c(\bm{\xi},\bm{\xi},\textbf v) +b(\textbf{v}, p_1)-b(\bm{\xi}, q)= (\textbf F, \textbf v), a.s. ,
	$$
	$$
	a(\bm{\eta},\textbf{v})+c(\bm{\eta}, \bm{\eta}, \textbf{v})+c(\bm{\eta},\bm{\xi},\textbf v)+c(\bm{\xi}, \bm{\eta},\textbf{v})+b(\textbf{v}, p_2)-b(\bm{\eta}, q)= \left(\sigma\frac{d\textbf W}{d\textbf x}, \textbf{v}\right), a.s..\\
	$$
	
	Note that the following weak formulation is equivalent to \eqref{w_sSNS1}-\eqref{w_sSNS2}: find $\bm{\eta},\bm{\xi}\in \mathbb V_{div}$ such that
	\begin{gather}
		a(\bm{\xi},\textbf v)+c(\bm{\xi},\bm{\xi},\textbf v) = (\textbf F, \textbf v), a.s. ,\label{w_sSNS3}\\
		a(\bm{\eta},\textbf{v})+c(\bm{\eta}, \bm{\eta}, \textbf{v})+c(\bm{\eta},\bm{\xi},\textbf v)+c(\bm{\xi}, \bm{\eta},\textbf{v}) = \left(\sigma\frac{d\textbf W}{d\textbf x}, \textbf{v}\right), a.s. .\label{w_sSNS4}
	\end{gather}
\end{remark}
\subsection{Properties of splitting method }
\begin{theorem}[Equivalence]\label{Equivalence}
	The solution of equation \eqref{SNS} is equivalent to the sum of the solutions of equations \eqref{SSNS1} and \eqref{SSNS2}  of the splitting method.
\end{theorem}
\begin{proof}  First, equations \eqref{SNS}, \eqref{SSNS1}, and \eqref{SSNS2} can be written as: 
	$$A: L[\textbf{u}]=0,\quad B: L_1[\bm{\xi}]=0,\quad C: L_2[\bm{\eta}]=0,$$
	respectively. Here, $L,L_1,L_2$ are operators corresponding to the respective equations, which map functions to zero.
	
	Suppose there exists an isomorphic transformation $T$ that relates $\bm{\xi}+\bm \eta$ to $\textbf{u}$ such that: $\textbf{u}=T[\bm{\xi}+\bm \eta]$, i.e. $\textbf{u}=T[\bm{\xi}+\bm \eta]=T_1[\bm{\xi}]+T_2[\bm \eta]$. Here, $T$ is a linear operator that transforms the combined random variables into the solution $\textbf{u}$, and $T_1$ and $T_2$ are its respective components acting on $\bm{\xi}$ and $\bm \eta$.
	
	Let us prove that $\textbf{u}=T[\bm{\xi}+\bm \eta]$ satisfies equation $A$.
	\begin{equation}
		\begin{split}
			L[\textbf{u}]=L[T[\bm{\xi}+\bm \eta]]=L[T_1[\bm{\xi}]+T_2[\bm \eta]]=LT_1[\bm{\xi}]+LT_2[\bm \eta],\nonumber
		\end{split}
	\end{equation}
	when $LT_1=L_1$ and $LT_2=L_2$, then $ L[\textbf{u}]=0$.
	
	Next, consider that when $A$ is true, $B$ and $C$ are satisfied.
	\begin{equation}
		\begin{split}
			L_1[\bm{\xi}]=L_1\left [T_1^{-1}\left[\textbf{u}-T_2[\bm \eta]\right]\right ]=L_1T_1^{-1}[\textbf{u}]-L_1T_1^{-1}T_2[\bm \eta],\\
			L_2[\bm \eta]=L_2\left [T_2^{-1}\left[\textbf{u}-T_1[\bm{\xi}]\right]\right ]=L_2T_2^{-1}[\textbf{u}]-L_2T_2^{-1}T_1[\bm{\xi}], \nonumber
		\end{split}
	\end{equation}
	where $T^{-1}$ denotes the inverse of the transformation $T$. When $LT_1=L_1$ and $LT_2=L_2$, then one of $L_1[\bm{\xi}], L_2[\bm{\eta}]$  is zero, the other can only be zero.
	
	In summary, the solutions are equivalent.\end{proof}

To further illustrate the source of error in statistical solutions, we first consider the relationship between the expectation of deterministic solution $\bm{\xi}$ in splitting methods and the expectation of the original equation $\textbf u.$

\begin{remark}\label{remark3_2}
	Assume there exists a constant $C_{c_\alpha}> 0$ satisfying the trilinear constraint $c(\textbf u_\alpha,\textbf u_\alpha,\textbf u_0 -\bm{\xi})\le C_{c_\alpha}|\textbf u_\alpha|_1^2|\textbf u_0 -\bm{\xi}|_1$, and a  constant $C_\Omega > 0$ satisfying the Poincaré inequality $\|\textbf{u}\| \le C_\Omega|\textbf{u}|_1$.  If the parameters satisfy 
	\begin{equation}
		\frac{C_\Omega}{\nu} \max_{\alpha \geq 1} \{C_{c_\alpha}\}\cdot \text{Var}(\textbf{u})\ll 1,
	\end{equation}
	then the expected value of the deterministic part $\bm{\xi}$ of the splitting method is very close to the expected value of the original equation \textbf{u}, i.e. $\mathbb E[\textbf{u}]$. Here, $\alpha$ is an integer representing the order of the expansion terms, and $\textbf u_\alpha$ are the coefficients corresponding to these higher-order terms, which capture the stochastic nature of the solution.
\end{remark}

\begin{proof} 
	Considering $\mathbb E\left[\sigma\frac{d\textbf W}{d\textbf x}\right]= 0$, according to equation \eqref{w_SNS1}, we have
	\begin{equation}
		\begin{split}
			a(\mathbb E\textbf{u},\textbf{v})+\mathbb E[c(\textbf{u},\textbf{u},\textbf{v})]+\mathbb E[(b(\textbf{v},p)]=\mathbb E[(\textbf{F},\textbf{v})].\nonumber
		\end{split}
	\end{equation}
	Let $\textbf H_\alpha,\textbf H_\beta$ denote multi-dimensional Hermite polynomials, where $\alpha$ and $\beta$ are integers. The coefficients $\textbf u_\alpha,\textbf u_\beta$ correspond to these higher-order terms, allowing us to express the solution as
	$$
	\textbf{u}=\textbf u_0+\sum_{\alpha \geq 1}{\textbf u_{\alpha}\textbf H_\alpha(w)}.$$
	
	By virtue of orthogonality, we have $\mathbb E[\textbf H_\alpha]=\mathbb E[\textbf H_\alpha \cdot \textbf H_0]=0  $ for $\alpha>1$. It follows that
	$$\textbf u_0 = \mathbb E[\textbf{u}].$$
	Consequently, we can express the expected value as
	$$
	\mathbb E[c(\textbf{u},\textbf{u},\textbf v_1)]= c(\textbf u_0,\textbf u_0,\textbf v_1)+ \sum_{\alpha,\beta \geq 1}c(\textbf u_\alpha,\textbf u_\beta,\textbf v_1)\mathbb E[\textbf H_\alpha \textbf H_\beta],
	$$	
	which can be rewritten as
	$$
	a(\textbf u_0,\textbf{v})
	+c(\textbf u_0,\textbf u_0,\textbf v)+\sum_{\alpha,\beta \geq 1}c(\textbf u_\alpha,\textbf u_\beta,\textbf v)\mathbb E[\textbf H_\alpha \textbf H_\beta]+b(\textbf{v},\mathbb E[p])
	=\mathbb E(\textbf{F},\textbf{v}).
	$$
	
	Besides, the equation for $\bm{\xi}$ in the splitting method is as follows 
	$$
	a(\textbf u_0,\textbf{v})
	+c(\textbf u_0,\textbf u_0,\textbf v)+b(\textbf{v},\mathbb E[p])
	=\mathbb E(\textbf{F},\textbf{v}).
	$$
	
	Based on \eqref{SNS}, \eqref{SSNS2} and \eqref{w_SNS1}, we are able to obtain
	$$
	a(\textbf{u} -\bm{\xi}, \textbf{v})+c(\textbf u, \textbf u,\textbf v)-c(\bm{\xi},\bm{\xi},\textbf v) = \left(\sigma\frac{d\textbf W}{d\textbf x}, \textbf{v}\right).
	$$
	
	Take expectation on both sides 
	$$
	a(\textbf u_0 -\bm{\xi}, \textbf{v})+\sum_{\alpha,\beta \geq 1}c(\textbf u_\alpha,\textbf u_\beta,\textbf v)\mathbb E[\textbf H_\alpha \textbf H_\beta] = 0.
	$$
	
	Therefore
	$$
	|a(\textbf u_0 -\bm{\xi}, \textbf{v})|  \leq \sum_{\alpha,\beta \geq 1}\left|c(\textbf u_\alpha,\textbf u_\beta,\textbf v)\mathbb E[\textbf H_\alpha \textbf H_\beta]\right|
	= \sum_{\alpha \geq 1}\left|c(\textbf u_\alpha,\textbf u_\alpha,\textbf v)\right|.
	$$
	
	Setting $\textbf{v}=\textbf u_0 -\bm{\xi},$
	\begin{equation}
		\begin{split}
			\frac{\nu}{C_\Omega}\|\textbf u_0 -\bm{\xi}\|^2 &\leq \nu|\textbf u_0 -\bm{\xi}|_1^2 = \left|a(\textbf u_0 -\bm{\xi},\textbf u_0 -\bm{\xi})\right| \nonumber\\
			& \leq \sum_{\alpha \geq 1}\left|c(\textbf u_\alpha,\textbf u_\alpha,\textbf u_0 -\bm{\xi})\right|\nonumber \\
			&\leq \sum_{\alpha \geq 1}C_{c_\alpha}\|\textbf u_\alpha\|^2\|\textbf u_0 -\bm{\xi}\|.\nonumber
		\end{split}
	\end{equation}
	Hence, we have
	\begin{equation}
		\begin{split}
			\frac{\nu}{C_\Omega}\|\textbf u_0 -\bm{\xi}\| 
			&\leq \sum_{\alpha \geq 1}C_{c_\alpha}\|\textbf u_\alpha\|^2 \nonumber \\
			&\leq \max_{\alpha \geq 1}\{C_{c_\alpha}\}\left(\mathbb E[\textbf{u}]^2-(\mathbb E{[\textbf{u}]})^2\right)=\max_{\alpha \geq 1}\{C_{c_\alpha}\}\text{Var}(\textbf{u}).\nonumber
		\end{split}
	\end{equation}
	That is to say,
	\begin{equation}
		\begin{split}
			\|\textbf u_0 -\bm{\xi}\| 
			\leq \frac{C_\Omega }{\nu}\max_{\alpha \geq 1}\{C_{c_\alpha}\}\text{Var}(\textbf{u}). \nonumber
		\end{split}
	\end{equation}
	Thus, we can conclude that if the right side of the inequality is very small, the expected value of the deterministic part $\bm{\xi}$ of the splitting method is very close to the expected value of the original equation $\textbf{u}$.  
\end{proof}


Next, we will analyze the advantages of the splitting method from the perspective of stability, existence and uniqueness of the solution.

\begin{lemma}[Stability of the splitting-solution] \label{lem3} 
	
	1) Let $(\bm{\xi},p_1)\in(\textbf{V}_{div}, \textbf{W}_{div})$ be any solution of  \eqref{SSNS2}, then 
	\begin{gather}
		|\bm{\xi}|_1\le \frac{1}{\nu}\|\textbf{F}\|
		_{ H^{-1} (\Omega)}, \label{l3}\\
		\|p_1\|_{\mathbb W}\le \frac{1}{\beta_1}\left(2\|\textbf{F}
		\|_{H^{-1} (\Omega)} +\frac{C_{c_3}}{\nu^2}\|\textbf{F}
		\|^2_{H^{-1} (\Omega)}\right),\label{l4}
	\end{gather}
	where the constant $C_{c_3} > 0$ such that $c(\bm{\xi},\bm{\xi},\textbf v)\le
	C_{c_3}|\bm{\xi}|_1^2|\textbf v|_1 .$
	
	2) Let $\bm{\eta} \in \mathbb V_{0}$ be any solution of  \eqref{SSNS2} and $\textbf{v}$ satisfies inequality \eqref{l3}-\eqref{l4}, then 
	\begin{gather}
		\|\bm{\eta}\|_{\mathbb V_0}  \le N_{\bm{\eta}}, \label{l5}\\
		\| p_2\|_{\mathbb W_0}\le  \frac{1}{\beta} \left(\mathbb E\left[\left \|\sigma\frac{d\textbf W}{d\textbf x}
		\right\|_{ H^{-1} (\Omega)} \right]+C_{c_1}N_{\bm{\eta}}^2+\left(\nu+\frac{C_{c_1}}{\nu}\|\textbf{F}\|
		_{ H^{-1} (\Omega)}\right)N_{\bm{\eta}}\right),\label{l6}
	\end{gather}
	where  	$N_{\bm{\eta}} = \displaystyle{\mathbb E\left[ \left\|\sigma\frac{d\textbf{W}}{d\textbf x}\right\|_{H^{-1}(\Omega)}\right]}\Big/\left({\nu- C_{c_4}\frac{\|\textbf{F}\|_{H^{-1}(\Omega)}}{\nu}}\right)$, the constant $ C_{c_4}>0 $ such that  $\mathbb E[c(\bm\eta,\bm{\xi},\bm \eta)]\le C_{c_4}\|\bm{\eta}\|_{\mathbb V_0}^2 \|\bm{\xi}\|_{\mathbb V_0}$.
\end{lemma}

\begin{proof} 1) Similar to \textbf{Lemma \ref{lem2}}, inequality \eqref{l3}-\eqref{l4} clearly holds. 2) We choose test function $(\textbf{w}, q)= (\bm{\eta}, p_2)$ in \eqref{w_sSNS2}
	\begin{equation}\label{23}
		\begin{split}
			\mathbb E[a(\bm{\eta}, \bm{\eta})]+\mathbb E\left[c(\bm{\eta}, \bm{\eta}, \bm{\eta})\right]+\mathbb E[c(\bm{\eta}, \bm{\xi}, \bm{\eta})]+\mathbb E\left[c(\bm{\xi}, \bm{\eta}, \bm{\eta})\right] = \mathbb E\left[\left(\sigma\frac{d\textbf W}{d\textbf x},\bm{\eta}\right)\right],
		\end{split}
	\end{equation}
	and observe that by \eqref{c_4}, one has
	\begin{equation}\label{24}
		\begin{split}
			\mathbb E[a(\bm{\eta}, \bm{\eta})]+\mathbb E[c(\bm{\eta}, \bm{\xi}, \bm{\eta})]
			= \mathbb E\left[\left(\sigma\frac{d\textbf W}{d\textbf x},\bm{\eta}\right)\right].\nonumber
		\end{split}
	\end{equation}
	Then it follows that
	$$
	\nu \|\bm{\eta}\|^2_{\mathbb V_0}=\mathbb E[a(\bm{\eta}, \bm{\eta})]=\left|\mathbb E\left[\left(\sigma\frac{d\textbf W}{d\textbf x},\bm{\eta}\right)\right]-\mathbb E[c(\bm{\eta}, \bm{\xi}, \bm{\eta})]\right|.
	$$
	We  get energy estimation by using the Cauchy-Schwarz inequality, that is to say
	for some $ C_b,C_{c_4} >0$, we have
	\begin{equation}
		\begin{split}
			\nu \|\bm{\eta}\|^2_{\mathbb V_0} &= \mathbb E[a(\bm{\eta}, \bm{\eta})]\\
			&\leq \left|\mathbb E\left[\left(\sigma\frac{d\textbf W}{d\textbf x},\bm{\eta}\right)\right]\right|+\left|\mathbb E[c(\bm{\eta}, \bm{\xi}, \bm{\eta})]\right|\\
			&\leq \mathbb E\left[\left\|\sigma\frac{d\textbf W}{d\textbf x}\right\|_{H^{-1}(\Omega)}\right]\|\bm{\eta}\|_{\mathbb V_0}+ C_{c_4}\frac{\|\textbf{F}\|
				_{H^{-1} (\Omega)}}{\nu} \|\bm{\eta}\|^2_{\mathbb V_0}. \nonumber
		\end{split}
	\end{equation}
	
	Naturally, we derive the inequity 
	$$
	(\nu- C_{c_4}\frac{\|\textbf{F}\|
		_{H^{-1} (\Omega)}}{\nu}) \|\bm{\eta}\|_{\mathbb V_0}  \le \mathbb E\left[\left\|\sigma\frac{d \textbf W}{d\textbf x}\right\|_{ H^{-1}(\Omega)}\right], 
	$$
	i.e., \eqref{l5} is valid. 
	
	According to the inf-sup condition and \eqref{dual}, one obtains for the pressure 
	\begin{equation}
		\begin{split}
			&\quad \| p_2\|_{\mathbb W_0}\\
			&\le  \frac{1}{\beta}\sup_{0\neq\textbf v\in \mathbb V_0}\frac{\mathbb E[b(\textbf{v}, p_2)]}{\|\textbf v\|_{\mathbb V_0}} \\
			&=\frac{1}{\beta}\sup_{0\neq\textbf v\in \mathbb V_0}\frac{\mathbb E\left[(\sigma\frac{d\textbf W}{d \textbf{x}}, \textbf{v})\right]-\mathbb E[a(\bm{\eta},\textbf{v})]-\mathbb E[c(\bm{\eta}, \bm{\eta}, \textbf{v})]-\mathbb E[c(\bm{\eta},\bm{\xi},\textbf v)]-\mathbb E[c(\bm{\xi}, \bm{\eta},\textbf{v})]}{\|\textbf v\|_{\mathbb V_0}} \nonumber,
		\end{split}
	\end{equation}
	namely,
	\begin{equation}
		\begin{split}
			\| p_2\|_{\mathbb W_0}\le  \frac{1}{\beta} \left(\mathbb E\left[\left \|\sigma\frac{d\textbf W}{d\textbf x}
			\right\|_{ H^{-1} (\Omega)} \right]+\nu\|\bm{\eta}\|_{\mathbb V_0}+C_{c_1}\|\bm{\eta}\|^2_{\mathbb V_0}+C_{c_1}\|\bm{\eta}\|_{\mathbb V_0}\|\bm{\xi}\|_{\mathbb V_0}\right).\nonumber
		\end{split}
	\end{equation}
	Substitute \eqref{l3} and \eqref{l5} into the inequality above and it is evident that inequality \eqref{l6} holds.\end{proof}

\begin{theorem}[The existence and uniqueness theorem of the splitting-solution]\label{thm2}
	
	Let $\Omega \in \mathbb{R}^2$ be a bounded domain with Lipschitz boundary $\partial \Omega$ and the random body force $\textbf{F}+\sigma\frac{d\textbf W}{d\textbf x} \in L^2 \left(\Theta; H^{-1}(\Omega)\right), $ if
	\begin{equation}\label{Th1}
		\frac{\|\textbf{F}\|_{L^2(\Omega)}}{\nu ^2}  \textless \frac{1}{C_{c_5}} \qquad \text{and }\qquad
		\frac{ N_1 N_{\eta} \nu}{\nu^2-N_2\|\textbf{F}\|
			_{H^{-1} (\Omega)}} \textless 1.
	\end{equation}
	where  $C_{c_5} > 0$ such that $c(\textbf{z}_1-\textbf{z}_2,
	\bm{\xi}_1,\bm{\xi}_1-\bm{\xi}_2) \le C_{c_5}|\textbf{z}_1-\textbf{z}_2|_1|\bm{\xi}_1|_1|\bm{\xi}_1-\bm{\xi}_2|_1, N_{\bm{\eta}}=\mathbb E\left[ \left\|\sigma\frac{d\textbf{W}}{d\textbf x}\right\|_{H^{-1}(\Omega)}\right]\Big/\left(\nu- C_{c_4}\frac{\|\textbf{F}\|
		_{H^{-1}(\Omega)}}{\nu}\right),$\quad$N_1=\sup\limits_{0\neq\hat{\bm{\eta}}-\hat{\bm{\eta}}_1,\bm{\eta},\bm{\eta}-\bm{\eta}_1\in \mathbb V_0} \mathbb E[|c(\hat{\bm{\eta}}-\hat{\bm{\eta}}_1,\bm{\eta},\bm{\eta}-\bm{\eta}_1)|]/\left(\|\hat{\bm{\eta}}-\hat{\bm{\eta}}_1\|_\mathbb V\|\bm{\eta}\|_\mathbb V\|\bm{\eta}-\bm{\eta}_1\|_\mathbb V\right)$ and $N_2=\sup\limits_{0\neq\textbf {v},\bm{\eta}-\bm{\eta}_1\in \mathbb V_0}\mathbb E[|c(\bm{\eta}-\bm{\eta}_1,\bm{\xi},\bm{\eta}-\bm{\eta}_1)|]/\left(\|\bm{\eta}-\bm{\eta}_1\|^2_\mathbb V\|\bm{\xi}\|_\mathbb V\right)$, then there exists a unique pair of stochastic processes $(\textbf u, p)\in (\mathbb V_0, \mathbb W_0)$  satisfying equations\eqref{w_SNS1} a.s..
\end{theorem}

\begin{proof} Since equation \eqref{w_sSNS3}-\eqref{w_sSNS4}is equivalent to \eqref{w_sSNS1}-\eqref{w_sSNS2}, so the equivalent problem \eqref{w_sSNS3}-\eqref{w_sSNS4} in the divergence-free subspace $\mathbb V_{div}$ is considered. 
	
	1) Firstly, the existence and uniqueness condition of the deterministic solution $\bm{\xi}$ in the splitting equation \eqref{w_sSNS3} is considered.
	
	Define the bilinear form
	$$
	A_z(\bm{\xi},\textbf v) := a(\bm{\xi},\textbf v)+c(\textbf z,\bm{\xi},\textbf v) ,
	$$
	for given $\textbf{z}\in \textbf V_{div}$. Clearly, we  obtain
	\begin{equation}\label{th1}
		|A_z(\bm{\xi},\textbf v)| \le C_{z}\|\bm{\xi}\|_1\|\textbf{v}\|_1 , 
	\end{equation}
	where $C_z$ is a constant related to $\textbf{z}.$
	$$
	A_z(\bm{\xi},\bm{\xi}) = a(\bm{\xi},\bm{\xi}) = \nu |\bm{\xi}|^2_1 ,
	$$
	so that $A_z(\bm{\xi},\textbf{v})$ is coercive. Moreover,  $\textbf{F}$ is a continuous linear functional on $\textbf V_{div}$. Therefore, by Lax-Milgram theorem, for any given $\textbf{z}\in \textbf V_{div},$ there exists a unique solution $\bm{\xi}\in \textbf V_{div}$  such that
	\begin{equation}\label{th3_4}
		A_z(\bm{\xi},\textbf v) = (\textbf{F}, \textbf{v}), \quad \forall \textbf v\in \textbf V_{div} ,
	\end{equation}
	that is to say
	\begin{equation}\label{30}
		a(\bm{\xi},\textbf v)+c(\textbf z,\bm{\xi},\textbf v) = \left(\textbf{F}, \textbf{v}\right) , \quad \forall \textbf v\in \textbf V_{div}.\nonumber
	\end{equation}
	Define the nonlinear operator $\pi: \textbf{z} \longrightarrow \textbf{v}$ through \eqref{th3_4}, it will be shown that $\pi$ is a contraction mapping. For any $\textbf{z}_1, \textbf{z}_2\in \textbf V_{div}$, there exists a $\gamma \in(0,1) $ such that 
	\begin{equation}
		|\pi(\textbf{z}_1)-\pi(\textbf{z}_2)|_1 = |\bm{\xi}_1-\bm{\xi}_2|_1 \le \gamma |\textbf{z}_1-\textbf{z}_2|_1, \nonumber
	\end{equation}
	where $\pi(\textbf{z}_1) = \bm{\xi}_1, \pi(\textbf{z}_2) = \bm{\xi}_2$, so there exists a unique solution $\textbf{z}\in \textbf V_{div}$  such that $\textbf{z} = \bm{\xi}.$ 
	
	Indeed, taking $\bm{\xi}_1,\bm{\xi}_2$ to satisfy the equation \eqref{th3_4}, we obtain
	\begin{equation}
		a(\bm{\xi}_1-\bm{\xi}_2, \textbf v )+c(\textbf{z}_1-\textbf{z}_2,
		\bm{\xi}_1,\textbf v)+ c(\textbf{z}_2,
		\bm{\xi}_1-\bm{\xi}_2,\textbf v)= 0.\nonumber
	\end{equation}
	Let $\textbf{v}=\bm{\xi}_1-\bm{\xi}_2$, the last term on the left-hand side vanishes because of \eqref{c_4} and one obtains with \eqref{c_2} and H\"older inequality
	\begin{align}
		\nu |\bm{\xi}_1-\bm{\xi}_2|^2_1&\le |c(\textbf{z}_1-\textbf{z}_2,
		\bm{\xi}_1,\bm{\xi}_1-\bm{\xi}_2)|  \nonumber\\
		&\le C_{c_5}|\textbf{z}_1-\textbf{z}_2|_1|\bm{\xi}_1|_1|\bm{\xi}_1-\bm{\xi}_2|_1 \nonumber \\
		&\le C_{c_5}(|\textbf{z}_1-\textbf{z}_2|^2_1|\bm{\xi}_1|^2_1)^\frac{1}{2}(|\bm{\xi}_1-\bm{\xi}_2|^2_1)^\frac{1}{2}. \nonumber
	\end{align}
	This, together with \eqref{l1} and  \eqref{th1}, implies that 
	\begin{equation}\label{34}
		|\bm{\xi}_1-\bm{\xi}_2|_1
		\le \frac{C_{c_5}}{\nu ^2}\|\textbf{F}\|_{H^{-1}(\Omega)}|\textbf{z}_1-\textbf{z}_2|_1.
	\end{equation}
	Naturally, if the following conditions are met, then $\pi$ is a contraction mapping
	\begin{equation}\label{35}
		\frac{C_{c_5}}{\nu ^2}\|\textbf{F}
		\|_{H^{-1}(\Omega)}\textless 1.
	\end{equation}
	Obviously, the first inequality in \eqref{Th1} holds.
	
	2) Next, we prove the existence and uniqueness condition of the complex stochastic solution $\bm{\eta}$ in the splitting equation \eqref{w_sSNS4}.
	Define the bilinear form
	$$
	A_s(\bm{\eta},\textbf w) := \mathbb E[a(\bm{\eta},
	\textbf{v})]+\mathbb E[c(\textbf{s}, \bm{\eta}, \textbf{v})]+\mathbb E[c(\bm{\eta},\bm{\xi},\textbf{v})]+\mathbb E[c(\bm{\xi}, \bm{\eta}, \textbf{v})].
	$$
	Clearly, we  obtain
	\begin{equation}
		|A_s(\bm{\eta},\textbf v)| \le \max\{C_{\bm{\xi}},C_\textbf{s}\}\|\bm{\eta}\|_1\|\textbf{v}\|_1, \nonumber
	\end{equation}
	where $C_{\bm{\xi}},C_\textbf{s}$ are  constants related to $\bm{\xi},\textbf{s}.$ Furthermore, we have
	\begin{align}\label{38}
		\nu \|\bm{\eta}\|^2_{\mathbb V_0}  &=\mathbb E[a(\bm{\eta}\nonumber,
		\bm{\eta})] \\
		&=A(\bm{\eta},\bm{\eta})-|\mathbb E[c(\textbf{s}, \bm{\eta}, \bm{\eta})]|-|\mathbb E[c(\bm{\eta},\bm{\xi},\bm{\eta})]|-|\mathbb E[c(\bm{\xi}, \bm{\eta}, \bm{\eta})]| \nonumber\\
		&\le A(\bm{\eta},\bm{\eta}) +\|\bm{\eta}\|_{\mathbb V_0}^2 \|\bm{\xi}\|_{\mathbb V_0 }. 
	\end{align}
	Now from Cauchy's inequality with $\epsilon ( \epsilon > 0)$, we observe
	\begin{equation}
		\|\bm{\eta}\|_{\mathbb V_0}^2 \|\bm{\xi}\|_{\mathbb V_0}\le \epsilon \|\bm{\eta}\|_{\mathbb V_0}^2+\frac{1}{4\epsilon}\|\bm{\xi}\|_{\mathbb V_0}.\nonumber
	\end{equation}
	Insert this estimate into \eqref{38} , we have
	\begin{equation}
		\nu \|\bm{\eta}\|^2_{\mathbb V_0} \le A(\bm{\eta},\bm{\eta})+ \epsilon \|\bm{\eta}\|_{\mathbb V_0}^2+\frac{1}{4\epsilon}\|\bm{\xi}\|_{\mathbb V_0}.\nonumber
	\end{equation}
	Thus, as $\|\bm{\eta}\|_{\mathbb V_0} \longrightarrow \infty$ and choose $\epsilon > 0$ so small that
	$$
	\frac{A(\bm{\eta},\bm{\eta})}{\|\bm{\eta}\|_{\mathbb V_0}}\geq \left(\nu-\epsilon\right) \|\bm{\eta}\|_{\mathbb V_0}-\frac{1}{4\epsilon}\frac{C_{\bm{\xi}}}{\|\bm{\eta}\|_{\mathbb V_0}}\longrightarrow \infty,
	$$
	i.e.,$A(\bm{\eta},\textbf{v})$ is coercive.
	Therefore, by Lax-Milgram theorem,  there exists a unique solution $\bm{\eta}\in \mathbb V_{div}$  such that
	\begin{equation}\label{s}
		\mathbb E[a(\bm{\eta},
		\textbf{v})]+\mathbb E[c(\textbf{s}, \bm{\eta}, \textbf{v})]+\mathbb E[c(\bm{\eta}, \bm{\xi},\textbf{v})]+\mathbb E[c(\bm{\xi}, \bm{\eta}, \textbf{v})]= \mathbb E\left(\sigma\frac{d\textbf W}{d\textbf x}, \textbf{v}\right) , \quad \forall \textbf v\in \mathbb V_{div}.
	\end{equation}
	Define the nonlinear map $\pi_1: \textbf{s} \longrightarrow \bm{\eta} $  through \eqref{s}. Next, we will prove that $\pi_1$ is a contraction mapping. For any $\textbf{s}_1, \textbf{s}_2\in \mathbb V_{div}$, there exists a $\gamma_1 \in(0,1) $ such that 
	\begin{equation}
		|\pi_1(\textbf{s}_1)-\pi_1(\textbf{s}_2)|_1 = |\bm{\eta}_1-\bm{\eta}_2|_1 \le \gamma_1 |\textbf{s}_1-\textbf{s}_2|_1, \nonumber
	\end{equation}
	where $\pi_1(\textbf{s}_1) = \bm{\eta}_1, \pi_1(\textbf{s}_2) = \bm{\eta}_2$, so there exists a unique solution $\textbf{s}\in \mathbb V_{div}$  such that $\textbf{s} = \bm{\eta}.$ 
	
	Taking $\bm{\eta}_1,\bm{\eta}_2$ to satisfy the equation \eqref{s}, we obtain
	\begin{align*}
		\mathbb E[a(\bm{\eta}_1-\bm{\eta}_2, \textbf{v})]+\mathbb E[c(\textbf{s}_1-\textbf{s}_2,\bm{\eta}_1,\textbf{v})]+\mathbb E[c(\textbf{s}_2,\bm{\eta}_1-\bm{\eta}_2,\textbf{v})]+\mathbb E[c(\bm{\eta}_1-\bm{\eta}_2,\bm{\xi},\textbf{v})]&
		\\
		+\mathbb E[c(\bm{\xi}, \bm{\eta}_1-\bm{\eta}_2,\textbf{v})]= 0.&
	\end{align*}	
	Let $\textbf{v}=\bm{\eta}_1-\bm{\eta}_2$, we have
	\begin{align}
		&\nu \|\bm{\eta}_1-\bm{\eta}_2\|_\mathbb W ^2 \nonumber\\
		=& \left|\mathbb E[c(\textbf{s}_1-\textbf{s}_2,\bm{\eta}_1,\bm{\eta}_1-\bm{\eta}_2)]+\mathbb E[c(\textbf{s}_2,\bm{\eta}_1-\bm{\eta}_2,\bm{\eta}_1-\bm{\eta}_2)]+\mathbb E[c(\bm{\eta}_1-\bm{\eta}_2,\bm{\xi},\bm{\eta}_1-\bm{\eta}_2)]\right.\nonumber\\
		&+\left.\mathbb E[c(\bm{\xi}, \bm{\eta}_1-\bm{\eta}_2,\bm{\eta}_1-\bm{\eta}_2)]\right| \nonumber\\
		=& \left|\mathbb E[c(\textbf{s}_1-\textbf{s}_2,\bm{\eta}_1,\bm{\eta}_1-\bm{\eta}_2)]\right|+\left|\mathbb E[c(\bm{\eta}_1-\bm{\eta}_2,\bm{\xi}, \bm{\eta}_1-\bm{\eta}_2)]
		\right| \nonumber\\
		\leq & N_1\|\textbf{s}_1-\textbf{s}_2\|_\mathbb V\|\bm{\eta}_1\|_\mathbb V\|\bm{\eta}_1-\bm{\eta}_2\|_\mathbb V+N_2\|\bm{\eta}_1-\bm{\eta}_2\|^2_\mathbb V\|\bm{\xi}\|_\mathbb V \nonumber\\
		\leq & N_1\|\textbf{s}_1-\textbf{s}_2\|_\mathbb V\|\bm{\eta}_1-\bm{\eta}_2\|_\mathbb V N_{\eta}+N_2\|\bm{\eta}-\bm{\eta}_1\|^2_\mathbb V\frac{\|\textbf{F}\|
			_{H^{-1} (\Omega)}}{\nu}.\nonumber
	\end{align}
	Obviously, we only need to consider the following equation
	\begin{equation}
		\left(\nu- \frac{N_2\|\textbf{F}\|
			_{H^{-1} (\Omega)}}{\nu}\right)\|\bm{\eta}_1-\bm{\eta}_2\|_\mathbb V \leq N_1\|\textbf{s}_1-\textbf{s}_2\|_\mathbb V N_{\eta}.\nonumber
	\end{equation}
	Naturally, if the following conditions are met, then $\pi_1$ is a contraction mapping
	$$
	\frac{ N_1 N_{\eta} \nu}{\nu^2-N_2\|\textbf{F}\|
		_{H^{-1} (\Omega)}} < 1. 
	$$
	Summing up, we finish the proof of  \textbf{Theorem \ref{thm2}}.\end{proof}


It is indeed true that the stability condition of the splitting method in \textbf{Lemma \ref{lem3}} differs from the stability condition of the original equation in \textbf{Lemma \ref{lem2}}. In \textbf{Lemma \ref{lem3}}, the stability of the splitting method is translated into the stability conditions required for solving a deterministic NS equation and a stochastic equation, i.e.  $\bm{\xi}$ and $\bm{\eta}$ need to satisfy their respective stability conditions. Upon comparison, it becomes apparent that  the splitting method allows for a more relaxed stability condition, enabling better handling of data with nonlinear characteristics.  Therefore, the splitting method facilitates effective and precise problem-solving in practical scenarios, ultimately leading to more stable splitting outcomes. This offers additional choices and possibilities for effectively solving real-world problems.
\subsection{Modified splitting format}

Notably, the right-hand  control quantity of stochastic equation \eqref{w_sSNS2} in the splitting method is a disturbance of the external force, and this disturbance is quite small compared to the external force in actual situations. In this subsection, we will discuss how the presence or absence of a nonlinear term $c(\bm{\eta}, \bm{\eta}, \textbf{v}) $ affects the solution.

Specifically, consider the weak formulations of \eqref{SSNS3} in the splitting method , using the following formula: find $\tilde{\bm{\eta}}\in \mathbb V_0,\bm{\xi}\in H^1_0(\Omega) , \tilde p_2\in\mathbb W_0$ such that
\begin{equation}\label{iw_sSNS2}
	\begin{split}
		\mathbb E[a(\tilde{\bm{\eta}},\textbf{v})] +\mathbb E[c(\tilde{\bm{\eta}},\bm{\xi},\textbf v)]+\mathbb E[c(\bm{\xi}, \tilde{\bm{\eta}},\textbf{v})]+&\mathbb E[b(\textbf{v}, \tilde p_2)] = \mathbb E\left[\left(\sigma\frac{d\textbf W}{d\textbf x}, \textbf{v}\right)\right],\\
		&\mathbb E[b(\tilde{\bm{\eta}}, q)]= 0,
	\end{split}
\end{equation}
for all test function $(\textbf{v},q)\in(\mathbb V_0, \mathbb W_0)$.

\begin{remark}[Stability of the modified splitting format]
	
	Based on the previous discussion, similar to Lemma \ref{lem3}, the stability of the solution derived from the modified splitting method aligns with the stability of the solution obtained through the splitting method.
\end{remark}
Now, let's delve into the analysis of the existence and uniqueness of the solution within the framework of the modified splitting method.
\begin{theorem}[The existence and uniqueness theorem of the modified splitting format]\label{ex_thm2}
	
	Assuming the first inequality in \eqref{Th1} holds and $\textbf{F}+\sigma\frac{d\textbf W}{d\textbf x} \in L^2 \left(\Theta; H^{-1}(\Omega)\right).$ Then the variational equation \eqref{iw_sSNS2} has a unique solution $\tilde{\bm{\eta}} \in \mathbb V_0$.
\end{theorem}

\begin{proof} The existence and uniqueness of solutions to deterministic equations are shown in \textbf {Theorem \ref{thm2}}. Next we focus on the existence and uniqueness of solutions to improved stochastic equations. 
	
	Define the bilinear form
	$$
	A(\tilde{\bm{\eta}},\textbf v) := \mathbb E[a(\tilde{\bm{\eta}},
	\textbf{v})]+\mathbb E[c(\tilde{\bm{\eta}}, \bm{\xi},\textbf{v})]+\mathbb E[c(\bm{\xi}, \tilde{\bm{\eta}}, \textbf{v})].
	$$
	Clearly, we  obtain
	\begin{equation}
		|A(\tilde{\bm{\eta}},\textbf v)| \le C_\textbf{v}\|\tilde{\bm{\eta}}\|_1\|\textbf{v}\|_1, \nonumber
	\end{equation}
	where $C_{\bm{\xi}}$ is a constant related to $\bm{\xi}.$ Furthermore, we have
	\begin{align}\label{44}
		\nu \|\tilde{\bm{\eta}}\|^2_{\mathbb V_0}  &=\mathbb E[a(\tilde{\bm{\eta}},
		\tilde{\bm{\eta}})]\nonumber \\
		&=A(\tilde{\bm{\eta}},\tilde{\bm{\eta}})-|\mathbb E[c(\tilde{\bm{\eta}},\bm{\xi},\tilde{\bm{\eta}})]|\nonumber \\
		&\le A(\tilde{\bm{\eta}},\tilde{\bm{\eta}}) +\|\tilde{\bm{\eta}}\|_{\mathbb V_0}^2 \|\bm{\xi}\|_{\mathbb V_0 }.
	\end{align}
	Now from Cauchy's inequality with $\epsilon ( \epsilon > 0)$, we observe
	\begin{equation}
		\|\tilde{\bm{\eta}}\|_{\mathbb V_0}^2 \|\bm{\xi}\|_{\mathbb V_0}\le \epsilon \|\tilde{\bm{\eta}}\|_{\mathbb V_0}^2+\frac{1}{4\epsilon}\|\bm{\xi}\|_{\mathbb V_0}.\nonumber
	\end{equation}
	Insert this estimate into \eqref{44} , we have
	\begin{equation}
		\nu \|\tilde{\bm{\eta}}|^2_{\mathbb V_0} \le A(\tilde{\bm{\eta}},\tilde{\bm{\eta}})+ \epsilon \|\tilde{\bm{\eta}}\|_{\mathbb V_0}^2+\frac{1}{4\epsilon}\|\bm{\xi}\|_{\mathbb V_0}.\nonumber
	\end{equation}
	Thus, as $\|\tilde{\bm{\eta}}\|_{\mathbb V_0} \longrightarrow \infty$ and choose $\epsilon > 0$ so small that
	$$
	\frac{A(\tilde{\bm{\eta}},\tilde{\bm{\eta}})}{\|\tilde{\bm{\eta}}\|_{\mathbb V_0}}\geq \left(\nu-\epsilon\right) \|\tilde{\bm{\eta}}\|_{\mathbb V_0}-\frac{1}{4\epsilon}\frac{C_{\bm{\xi}}}{\|\tilde{\bm{\eta}}\|_{\mathbb V_0}}\longrightarrow \infty,
	$$
	i.e.,$A(\tilde{\bm{\eta}},\textbf{v})$ is coercive.
	
	Therefore, by Lax-Milgram theorem,  there exists a unique solution $\bm{\eta}\in \mathbb V_{div}$  such that
	\begin{equation}
		\begin{split}
			\mathbb E[a(\tilde{\bm{\eta}},
			\textbf{v})]+\mathbb E[c(\tilde{\bm{\eta}}, \bm{\xi},\textbf{v})]+\mathbb E[c(\bm{\xi}, \tilde{\bm{\eta}}, \textbf{v})]= \mathbb E\left(\sigma\frac{d\textbf W}{d\textbf x}, \textbf{v}\right) , \quad \forall \textbf v\in \mathbb V_{div}.\nonumber
		\end{split}
	\end{equation}
	
	Summing up, we finish the proof of \textbf{Theorem \ref{ex_thm2}}.\end{proof}

In summary, the stability of the solution obtained by the modified splitting method is related to the stability of the splitting method, and the condition for existence and uniqueness of the solution will be relaxed accordingly. In order to comprehensively understand the effects of ignoring nonlinear terms and to determine the stability and accuracy of the modified splitting method in providing existence and unique solutions, further analysis of statistical errors will be conducted.

\section{Spatial discretization}\label{sec4}

This chapter is dedicated to the discussion of error estimation in the finite element method. We employ the Galerkin finite element discretization method to handle equations \eqref{w_sSNS1}, \eqref{w_sSNS2} and \eqref{iw_sSNS2} , and linearize the non-linear convection terms using the Newton method. Furthermore, our primary focus is on analyzing the error estimation of the splitting method and the modified splitting format within the framework of steady-state SNS equation.

Let $\tau^h$ be the regular triangular finite element mesh of the domain $\Omega$, where $h$ represents the maximum width of all typical triangular finite element meshes. Define the finite element spaces
\begin{align}
	\textbf{V}^h &=\left\{ \textbf{v}^h=[v^h_1,v^h_2]^T : v_i^h\in H^1_0(\overline{\Omega}), 
	\textbf{v}^h|_e \in \mathcal{P}^m \times \mathcal{P}^m, \forall e \in \tau^h \right\} \label{V_1},\\
	\textbf{V}_0^h &=\left\{ \textbf{v}^h \in\textbf{V}^h : v_i^h = 0 \text{ on }  \partial \Omega, i=1,2 \right\}\label{V_0}, \\
	\textbf{Q}^h &=\left\{ q^h : q^h\in H^1_0({\overline{\Omega}}), q^h|_e \in \mathcal{P}^s , \forall e \in \tau^h \right\}, \label{Q}\\
	\textbf{V}_{div}^h &=\left\{\textbf{v}^h\in \textbf{V}^h: b(\textbf{v}^h,  q^h) =0,  \quad\forall q^h \in \textbf{Q}^h \right\}, \label{V_div^h}
\end{align}
where $P^m$ denotes the polynomial space with degree less than or equal to m. Consider $\textbf{V}^h\times \textbf{Q}^h$ as a pair of Taylor-Hood finite element spaces, namely $m=2$ and $s=1$. 

We also set
$$
\mathbb V^h :=L^2\left(\Theta; \textbf{V}^h\right),\quad  \mathbb W^h :=L^2\left(\Theta; \textbf{Q}^h\right).
$$
The Taylor-Hood finite element ensures the stability of the finite element discretization , which satisfies the following discrete inf-sup condition \eqref{lem1} or Ladyzhenskaya-Babuska-Brezzi(LBB) condition\cite{Gunzburger_1989}, i.e.,there exits a constant $\hat{\beta} > 0$ such that
\begin{equation}\label{D_LBB}
	\inf_{0\neq q^h\in \mathbb W_0^h}\sup_{0\neq\textbf{v}^h\in \mathbb V^h}\frac{\mathbb E[b(\textbf{v}^h, q^h)]}{\|\textbf{v}^h\|_{\mathbb V_0}\|q^h\|_{\mathbb W}} \geq
	\hat{\beta} .
\end{equation}

Then, the Galerkin finite element discretization of the steady-state SNS equation reads as follows: Find $(\textbf {u}^h,p^h)\in\mathbb V^h\times \mathbb Q^h$ such that
\begin{equation}\label{dw_SNS1}
	\begin{split}
		a(\textbf{u}^h, \textbf{v}^h)+c(\textbf{u}^h,\textbf{u}^h, \textbf{v}^h) +b(\textbf{v}^h, p^h) &= (\textbf{F}, \textbf{v}^h)+\left(\sigma\frac{d\textbf W}{d\textbf x}, \textbf{v}^h\right), a.s.\\
		b(\textbf{u}^h, q^h)&= 0, a.s.
	\end{split}
\end{equation}
for all test function $(\textbf{v}^h,q^h)\in\mathbb V^h\times \mathbb Q^h$.

Accordingly, the Galerkin finite element discrete format of the splitting method  is formulated as following:  Find $\bm{\eta}^h\in\mathbb V^h, \bm{\xi}^h\in\textbf{V}^h, p_1^h, p_2^h\in\mathbb W^h$ such that
\begin{equation}\label{D_sSNS1}
	\begin{split}
		\mathbb E[a(\bm{\xi}^h, \textbf v^h)]+\mathbb E[c(\bm{\xi}^h, \bm{\xi}^h, \textbf v^h)] +\mathbb E[b(\textbf v^h, p_1^h)] &=\mathbb E\left[\left(\textbf F, \textbf{v}^h\right)\right],\\
		\mathbb E[b(\bm{\xi}^h, q^h)]&= 0,
	\end{split}
\end{equation}
and
\begin{align}\label{D_sSNS2}
	\mathbb E[a(\bm{\eta}^h,\textbf{v}^h)] +\mathbb E[c(\bm{\eta}^h, \bm{\eta}^h, \textbf{v}^h)] +\mathbb E[c(\bm{\eta}^h,\bm{\xi}^h, \textbf v^h)]&+\mathbb E[c(\bm{\xi}^h, \bm{\eta}^h,\textbf{v}^h)]+\mathbb E[b(\textbf v^h, p_2^h)]&\nonumber\\
	&= \mathbb E\left[\left(\sigma\frac{d\textbf W}{d\textbf x}, \textbf{v}^h\right)\right],\nonumber\\
	\mathbb E[b(\bm{\eta}^h, q^h)]&= 0,
\end{align}
for all test function $(\textbf{v}^h,q^h)\in\mathbb V^h\times \mathbb Q^h$.

The existence and uniqueness of the discrete solution can be proven in a similar manner as for the continuous equation \cite{Powell_2012}. In particular, uniqueness is guaranteed only in the case of small external forces (on the right hand side) and large viscosity. 
\begin{lemma} [Stability of the Finite Element Splitting-solution]\label{lem5}
	
	1) Let $(\bm{\xi}^h,p_1^h)\in\textbf V^h\times \textbf Q^h$ be any solution of  \eqref{D_sSNS1}, then 
	\begin{gather}
		|\bm{\xi}^h|_1\le \frac{1}{\nu}\|\textbf{F}\|
		_{H^{-1} (\Omega)}, \label{l4_1}\\
		\| p_1^h\|_\mathbb W \le \frac{1}{\beta_1^h}\left(2\|\textbf{F}
		\|_{H^{-1} (\Omega)} +\frac{\hat{C}_{c_3}}{\nu^2}\|\textbf{F}
		\|^2_{H^{-1} (\Omega)}\right),\label{l4_2}
	\end{gather}
	where the constant $\hat{C}_{c_3} > 0$ such that $|c(\bm{\xi}^h, \bm{\xi}^h, \bm{\xi}^h)|\le \hat{C} _{c_3}|\bm{\xi}^h|_1^2|\textbf{v}^h|_1$.
	
	2) Let $\bm{\eta} \in \mathbb V^h$ be any solution of  \eqref{D_sSNS2} and $\bm{\xi}^h$ satisfies inequality \eqref{l4_1}-\eqref{l4_2}, then 
	\begin{gather}
		\|\bm{\eta}^h\|_{\mathbb V_0} \le \hat{N}_{\eta}, \label{l4_3}\\
		\| p_2^h\|_{\mathbb W_0}\le  \frac{1}{\beta^h} \left(\mathbb E\left[\left \|\sigma\frac{d\textbf W}{d\textbf x}
		\right\|_{ H^{-1} (\Omega)} \right]+C_{c_1}\hat{N}_{\bm{\eta}}^2+\left(\nu+\frac{C_{c_1}}{\nu}\|\textbf{F}\|
		_{ H^{-1} (\Omega)}\right)\hat{N}_{\bm{\eta}}\right),\label{l4_4}
	\end{gather}
	where $\hat{N}_{\eta}=\mathbb E\left[\left\|\sigma\frac{d\textbf{W}} {d\textbf x}\right\|_{H^{-1}(\Omega)}\right]\Big/\left(\nu- \hat{C}_{c_4}\frac{\|\textbf{F}\|
		_{H^{-1} (\Omega)}}{\nu}\right)$, the constant $\hat{C}_{c_4}>0$ such that $c(\bm\eta^h,\bm{\xi}^h,\bm \eta^h)\le \hat{C}_{c_4}|\bm \eta|_1^2|\bm{\xi}|_1$.
\end{lemma}

\begin{theorem}[The existence and uniqueness theorem of the finite element Splitting-solution]\label{thm4_2}
	
	Let $\Omega \in \mathbb{R}^2$ be a bounded domain with Lipschitz boundary $\partial \Omega$ and the random body force $\textbf{F}+\sigma\frac{d\textbf W}{d\textbf x} \in L^2 \left(\Theta; H^{-1}(\Omega)\right). $ If a pair of finite element spaces  $\mathbb V^h\times \mathbb Q^h$  is used that satisfies the discrete inf-sup condition \eqref{D_LBB}  and let in addition
	\begin{equation}\label{ex_solution}
		\begin{split}
			\frac{\|\textbf{F}\|_{L^2(\Omega)}}{\nu ^2}  \textless \frac{1}{\hat{C}_{c_5}} \qquad  \text{and }\qquad \frac{\hat{N}_1 \hat{N}_{\eta} \nu}{\nu^2-\hat{N}_2\|\textbf{F}\|
				_{H^{-1} (\Omega)}} \textless 1,
		\end{split}
	\end{equation}
	where $\hat{C}_{c_5},\hat{N_1},\hat{N_2}>0$ such that $c(\textbf{z}_1^h-\textbf{z}_2^h,\bm{\xi}_1^h,\bm{\xi}_1^h-\bm{\xi}_2^h) \le \hat{C}_{c_5}|\textbf{z}_1^h-\textbf{z}_2^h|_1|\bm{\xi}_1^h|_1|\bm{\xi}_1^h-\bm{\xi}_2^h|_1,\hat{N}_1=\sup\limits_{0\neq\hat{\bm{\eta}}^h-\hat{\bm{\eta}}^h_1,\bm{\eta}^h,\bm{\eta}^h-\bm{\eta}^h_1\in \mathbb V_0} \mathbb E[|c(\hat{\bm{\eta}}^h-\hat{\bm{\eta}}^h_1,\bm{\eta}^h,\bm{\eta}^h-\bm{\eta}^h_1)|]/\left(\|\hat{\bm{\eta}}^h-\hat{\bm{\eta}}^h_1\|_\mathbb V\|\bm{\eta}^h\|_\mathbb V\|\bm{\eta}^h-\bm{\eta}^h_1\|_\mathbb V\right)$ and $\hat{N}_2=\sup\limits_{0\neq\textbf {v}^h,\bm{\eta}^h-\bm{\eta}^h_1\in \mathbb V_0}\mathbb E[|c(\bm{\eta}^h-\bm{\eta}^h_1,\bm{\xi}^h,\bm{\eta}^h-\bm{\eta}^h_1)|]/\left(\|\bm{\eta}^h-\bm{\eta}^h_1\|^2_\mathbb V\|\bm{\xi}^h\|_\mathbb V\right)$, then there exists a unique pair of stochastic processes $(\textbf u^h, p^h)\in (\mathbb V_0, \mathbb W_0)$  satisfying equations\eqref{D_sSNS1}-\eqref{D_sSNS2} a.s..
	
\end{theorem}
\begin{theorem}[Finite Element Error Estimate of the Gradient of the Velocity] \label{thm4_1}
	
	Let $\Omega \in \mathbb{R}^2$ be a bounded domain with Lipschitz boundary $\partial \Omega$ and the random body force $\textbf{F}+\sigma\frac{d\textbf W}{d\textbf x} \in L^2 \left(\Theta; H^{-1}(\Omega)\right). $ If a pair of finite element spaces  $\mathbb V^h\times \mathbb Q^h$  is used that satisfies the discrete inf-sup condition \eqref{D_LBB}  and let in addition \eqref{ex_solution} be satisfied, then there exists a unique pair of stochastic processes $(\textbf u^h, p^h)\in (\mathbb V^h\times \mathbb Q^h)$  satisfying equations\eqref{D_sSNS1}-\eqref{D_sSNS2} . Then, the following error estimate holds 
	\begin{align}
		&\quad\left\|\textbf{u}-\textbf{u}^h\right\|_{\mathbb V} \nonumber\\
		&\leq C\left(\left(1+\frac{\mathbb E\left[\left\|\textbf F+ \sigma\frac{d\textbf W}{d\textbf x}\right\|_{H^{-1}(\Omega)}\right]}{\nu^2}\right)\inf\limits_{\textbf{v}^h \in\textbf{V}^h_{div}}\left\|\textbf{u}-\textbf{v}^h\right\|_{\mathbb V}+\frac{1}{\nu}\inf\limits_{q^h \in\textbf{Q}^h}\left\|p-q^h \right\|\right),\\
		&\| p-p^h\|_{\mathbb W} \le \frac{C}{\beta_1^h}\nu\left(1+\frac{\|\textbf{F}\|
			_{H^{-1} (\Omega)}}{\nu^2}\right)\left\|\textbf{u}-\textbf{u}^h\right\|_{\mathbb V}+\left(1+\frac{C}{\beta_1^h}\right)\left\|p-q^h \right\|_{\mathbb W}\label{p_error}.
	\end{align}
	
\end{theorem}

We omit the proof to save space and refer the interested reader to \cite{Hou_2006} for a detailed proof.
\subsection{Statistical error of the splitting method }

In this section, we will illustrate the accuracy and robustness of this splitting theory. In the finite element error analysis of the steady-state Navier-Stokes equations,  an estimate of the nonlinear (trilinear) term is necessary. For the splitting theory, the original problem is decomposed into a deterministic NS equation and a stochastic  equation. This decomposition allows us to focus on the influence of nonlinear terms primarily in the stochastic equation and analyze their behavior independently. The deterministic equation captures the main behavior of the system, while the stochastic equation accounts for the random fluctuations or uncertainties.

The following theorem establishes the error estimate for the solutions of the splitting stochastic NS equations.
\begin{theorem}[Error Estimation for the  Splitting Method]\label{thm4_4}
	Let $\Omega \in \mathbb{R}^2$ be a bounded domain with polyhedral and Lipschitz continuous boundary, let \eqref{Th1} be fulfilled, and let instead of \eqref{ex_solution} the stronger condition  
	\begin{equation}\label{stronger condition}
		\frac{N_{\eta}}{\nu}+\frac{\|\textbf{F}\|_{H^{-1} (\Omega)}}{\nu^2}\le \frac{7}{8} .   
	\end{equation}
	Meanwhile, $\bm{\xi}^h, p_1^h, p_2^h, \bm{\eta}^h$ be the unique solution of the Navier-Stokes equations \eqref{D_sSNS1}-\eqref{D_sSNS2}. Assume that this problem is discretized with inf-sup stable finite element spaces. Then, the following error estimate holds
	\begin{flalign}
		\left\|\bm{\xi}-\bm{\xi}^h\right\|_{\mathbb V} &\leq C\left\{\left(1+\frac{\|\textbf{F}\|
			_{H^{-1} (\Omega)}}{\nu^2}\right)\inf\limits_{\textbf{v}^h \in\textbf{V}^h_{div}}\left\|\bm{\xi}-\textbf{v}^h\right\|_{\mathbb V}+\frac{1}{\nu}\inf\limits_{q^h \in\textbf{Q}^h}\left\|p_1-q^h \right\|\right\},\\
		\left\| p_1-p_1^h\right\|_{\mathbb W} &\le \frac{C}{\beta_1^h}\nu\left(1+\frac{\|\textbf{F}\|
			_{H^{-1} (\Omega)}}{\nu^2}\right)\left\|\bm{\xi}-\bm{\xi}^h\right\|_{\mathbb V}+\left(1+\frac{C}{\beta_1^h}\right)\left\|p_1-q^h \right\|_{\mathbb W},\\
		\left\|\bm{\eta}-\bm{\eta}^h\right\|_{\mathbb V}& \le C\left\{\inf\limits_{\textbf{v}^h \in\textbf{V}^h_{div}}\left\|\bm{\eta}-\textbf{v}^h\right\|_{\mathbb V}+\frac{N_{\eta}N_{\bm{\xi}-\bm{\xi}^h}}{\nu}  \right\},\\
		\| p_2-p_2^h\|_{\mathbb W} &\le \frac{C}{\beta_1^h} N_{\eta}\left\|\bm{\xi}-\bm{\xi}^h\right\|_{\mathbb V}+\frac{C}{\beta_1^h}\nu\left(\nu+\frac{\|\textbf{F}\|
			_{H^{-1} (\Omega)}}{\nu}+N_{\eta}\right)\left\|\bm{\eta}-\bm{\eta}^h\right\|_{\mathbb V}\nonumber\\
		&\quad+\left(1+\frac{C}{\beta_1^h}\right)\left\|p_2-q^h \right\|_{\mathbb W},
	\end{flalign}
	where the constant $N_{\bm{\xi}-\bm{\xi}^h}=C\left\{\left(1+\frac{\|\textbf{F}\|
		_{H^{-1} (\Omega)}}{\nu^2}\right)\inf\limits_{\textbf{v}^h \in\textbf{V}^h_{div}}\left\|\bm{\xi}-\textbf{v}^h\right\|_{\mathbb V}+\frac{1}{\nu}\inf\limits_{q^h \in\textbf{Q}^h}\left\|p_1-q^h \right\|\right\}$ and the constant C does not depend on the mesh size.
\end{theorem}
\begin{proof} To establish the error estimate, we use test function $\textbf{v}^h \in \textbf{V}^h_{div}$ and subtract the discrete formulation  from
	the weak formulation of the steady-state NS equation
	\begin{gather}
		a(\bm{\xi}-\bm{\xi}^h, \textbf{v}^h)+c(\bm{\xi},\bm{\xi}, \textbf{v}^h) -c(\bm{\xi}^h,\bm{\xi}^h, \textbf{v}^h)= 0, a.s.\label{e_1} \\
		a(\bm{\eta}-\bm{\eta}^h,\textbf{v}^h)+c(\bm{\eta}, \bm{\eta}, \textbf{v}^h)-c(\bm{\eta}^h, \bm{\eta}^h, \textbf{v}^h)+c(\bm{\eta},\bm{\xi}, \textbf{v}^h)-c(\bm{\eta}^h,\bm{\xi}^h,\textbf{v}^h)\nonumber\\
		\qquad\qquad\qquad+c(\bm{\xi}, \bm{\eta},\textbf{v}^h) -c(\bm{\xi}^h, \bm{\eta}^h,\textbf{v}^h) = 0, a.s.\label{e_2}
	\end{gather}
	1) To deal with error equation \eqref{e_1}, we decompose the error into the approximate error and the discrete remainder
	\begin{equation}\label{e_3}
		\bm{\xi}-\bm{\xi}^h = \left(\bm{\xi}-I^h\bm{\xi}\right)- \left(\bm{\xi}^h-I^h\bm{\xi}\right)=\bm{\psi}_1-\bm{\psi}_2^h,    \qquad I^h\bm{\xi}\in V^h_{div}.
	\end{equation}
	Inserting this decomposition in the error equation and setting  $\textbf{v}^h=\bm{\psi}_2^h$ leads to
	\begin{equation}\label{e_4}
		\nu\left\|\nabla \bm{\psi}_2^h \right\|^2= \nu\left(\nabla\bm{\psi}_1,\nabla\bm{\psi}_2^h\right)+c(\bm{\xi}, \bm{\xi}, \bm{\psi}_2^h)-c(\bm{\xi}^h, \bm{\xi}^h, \bm{\psi}_2^h) -b(\bm{\psi}_2^h, p_1-q^h).
	\end{equation}
	Obviously, using the Young's inequality, we obtain
	$$
	b(\bm{\psi}_2^h, p_1-q^h)=(\nabla\cdot\bm{\psi}_2^h, p_1-q^h)\leq \frac{2}{\nu}\left\|p_1-q^h \right\|^2+ \frac{\nu}{8} \left\|\nabla \bm{\psi}_2^h \right\|^2,
	$$
	and
	\begin{equation}\label{e_5}
		\nu\left(\nabla\bm{\psi}_1,\nabla\bm{\psi}_2^h\right)\leq 2\nu\left\|\nabla\bm{\psi}_1 \right\|^2+ \frac{\nu}{8} \left\|\nabla \bm{\psi}_2^h \right\|^2.
	\end{equation}
	The first two trilinear terms are estimated as follows
	\begin{align} \label{e_6}
		&\quad c(\bm{\xi}, \bm{\xi}, \bm{\psi}_2^h)-c(\bm{\xi}^h, \bm{\xi}^h, \bm{\psi}_2^h)\nonumber\\
		&=c\left(\bm{\xi}, \bm{\xi}-\bm{\xi}^h, \bm{\psi}_2^h\right)+c\left(\bm{\xi}-\bm{\xi}^h, \bm{\eta}^h, \bm{\psi}_2^h\right)\nonumber\\
		&=c\left(\bm{\xi}, \bm{\psi}_1, \bm{\psi}_2^h\right)-c\left(\bm{\xi},\bm{\psi}_2^h,\bm{\psi}_2^h\right)+c\left(\bm{\psi}_1,\bm{\xi}^h,\bm{\psi}_2^h\right)-c\left(\bm{\psi}_2^h, \bm{\xi}^h,\bm{\psi}_2^h\right)\nonumber\\
		&=c\left(\bm{\xi}, \bm{\psi}_1,\bm{\psi}_2^h\right)+c\left(\bm{\psi}_1,\bm{\xi}^h,\bm{\psi}_2^h\right)-c\left(\bm{\psi}_2^h,\bm{\xi}^h,\bm{\psi}_2^h\right).
	\end{align}
	For the first term in equation \eqref{e_6}, using \eqref{e_5} yields
	\begin{align} 
		c\left(\bm{\xi}, \bm{\psi}_1,\bm{\psi}_2^h\right)&\leq C_{c_6}\left\|\nabla\bm{\xi}\right\|\left\|\nabla\bm{\psi}_1 \right\| \left\|\nabla \bm{\psi}_2^h \right\|\nonumber \\
		& \leq \frac{2C_{c_6}}{\nu}\left\|\nabla\bm{\psi}_1 \right\|^2\left\|\nabla\bm{\xi}\right\|^2+ \frac{\nu}{8} \left\|\nabla \bm{\psi}_2^h \right\|^2\nonumber \\
		& \leq \frac{C_{c_6}}{\nu^3}\left\|\nabla\bm{\psi}_1 \right\|^2\|\textbf{F}\|^2
		_{H^{-1} (\Omega)}+ \frac{\nu}{8} \left\|\nabla \bm{\psi}_2^h \right\|^2 .\label{e_7}
	\end{align}
	The estimation for the second term of equation \eqref{e_6} is performed analogously, yielding
	\begin{align}
		&\quad c\left(\bm{\psi}_1,\bm{\xi}^h,\bm{\psi}_2^h\right)\nonumber\\
		&\leq \frac{2C_{c_7}}{\nu}\left\|\nabla\bm{\psi}_1 \right\|^2\left\|\nabla\bm{\xi}^h\right\|^2+ \frac{\nu}{8} \left\|\nabla \bm{\psi}_2^h \right\|^2 \nonumber\\
		&\leq \frac{C_{c_7}}{\nu^3}\left\|\nabla\bm{\psi}_1 \right\|^2\|\textbf{F}\|^2
		_{H^{-1} (\Omega)}+ \frac{\nu}{8} \left\|\nabla \bm{\psi}_2^h \right\|^2,
	\end{align}
	and
	$$
	c\left(\bm{\psi}_2^h,\bm{\xi}^h,\bm{\psi}_2^h\right)\leq C_{c_8}\left\|\nabla\bm{\xi}^h\right\|\left\|\nabla \bm{\psi}_2^h \right\|^2 \leq \frac{4C_{c_8}\|\textbf{F}\|
		_{\textbf H^{-1}(\Omega)}}{\nu^2}\left(\frac{\nu}{4} \left\|\nabla \bm{\psi}_2^h \right\|^2\right)\leq \frac{\nu}{4} \left\|\nabla \bm{\psi}_2^h \right\|^2.
	$$
	Substitute the above error estimates into equation \eqref{e_4}, we get
	$$
	\left\|\nabla \bm{\psi}_2^h \right\|^2 \leq C\left(\left\|\nabla\bm{\psi}_1 \right\|^2+\frac{1}{\nu^4}\left\|\nabla\bm{\psi}_1 \right\|^2\|\textbf{F}\|^2
	_{H^{-1} (\Omega)}+\frac{1}{\nu^2}\left\|p_1-q^h \right\|^2\right).
	$$
	Finally, the error estimation of equation \eqref{e_1} is drawn by using the triangle inequality
	\begin{equation} \label{e_8}
		\begin{split}
			\left\|\bm{\xi}-\bm{\xi}^h\right\|_{\mathbb V}&\leq \left\|\nabla\bm{\psi}_1\right\|+\left\|\nabla\bm{\psi}_2^h\right\|\leq C\left(\left\|\nabla\bm{\psi}_1 \right\|+\frac{\|\textbf{F}\|
				_{H^{-1} (\Omega)}}{\nu^2}\left\|\nabla\bm{\psi}_1 \right\|+\frac{1}{\nu}\left\|p_1-q^h \right\|\right)\\
			&\leq C\left\{\left(1+\frac{\|\textbf{F}\|
				_{H^{-1} (\Omega)}}{\nu^2}\right)\inf\limits_{\textbf{v}^h \in\bm{\xi}^h_{div}}\left\|\nabla(\bm{\xi}-\textbf{v}^h)\right\|+\frac{1}{\nu}\inf\limits_{q^h \in\textbf{Q}^h}\left\|p_1-q^h \right\|\right\}.
		\end{split}
	\end{equation}
	On the other hand, we consider the following  finite element error estimate for the pressure.
	Observe that by the triangle inequality, one has
	\begin{equation}
		\|p_1-p^h\|_{\mathbb W}\le \|p_1-q^h\|_{\mathbb W} +\|p_1^h-q^h\|_{\mathbb W}.\label{p}
	\end{equation}
	Next, we handle the term $\|p_1^h-q^h\|_{\mathbb W}$. By the discrete inf-sup condition \eqref{D_LBB} and the insertion of the finite element problem \eqref{D_sSNS1} as well as the variational form of the steady-state Navier-Stokes equations \eqref{w_sSNS1}, we have
	\begin{equation}
		\begin{split}
			&\quad \| p_1^h-q^h\|_{\mathbb W}\\
			&\le  \frac{1}{\beta_1^h}\sup_{0\neq\textbf v^h\in \mathbb V^h_0}\frac{\mathbb E[b(\textbf{v}^h, p_1^h-q^h)]}{\|\textbf v^h\|_{\mathbb V^h_0}} \\
			&=\frac{1}{\beta_1^h}\sup_{0\neq\textbf v^h\in \mathbb V^h_0}\frac{\mathbb E[a(\bm{\xi}-\bm{\xi}^h,\textbf v^h)]-\mathbb E[b(\textbf{v}^h, p_1-q^h)]+\mathbb E[c(\bm{\xi}, \bm{\xi},\textbf v^h)]-\mathbb E[c(\bm{\xi}^h, \bm{\xi}^h,\textbf v^h)]}{\|\textbf v^h\|_{\mathbb V^h_0}}\\
			&=\frac{1}{\beta_1^h}\sup_{0\neq\textbf v^h\in \mathbb V^h_0}\frac{\mathbb E[a(\bm{\xi}-\bm{\xi}^h,\textbf v^h)]-\mathbb E[b(\textbf{v}^h, p_1-q^h)]+\mathbb E[c(\bm{\xi}, \bm{\xi}-\bm{\xi}^h,\textbf v^h)]-\mathbb E[c(\bm{\xi}-\bm{\xi}^h, \bm{\xi}^h,\textbf v^h)]}{\|\textbf v^h\|_{\mathbb V^h_0}}\label{p^h1}.
		\end{split}
	\end{equation}
	Now, we use the Cauchy-Schwarz inequality
	to yield that
	\begin{align}
		&\quad\| p_1^h-q^h\|_{\mathbb W}\nonumber\\
		&\le  \frac{1}{\beta_1^h}\left(\nu\|\bm{\xi}-\bm{\xi}^h\|_{\mathbb V^h_0}+C\|\bm{\xi}\|_{\mathbb V^h_0}\|\bm{\xi}-\bm{\xi}^h\|_{\mathbb V^h_0}+C\|\bm{\xi}^h\|_{\mathbb V^h_0}\|\bm{\xi}-\bm{\xi}^h\|_{\mathbb V^h_0}+\|p_1-q^h\|_{\mathbb W} \right)   \nonumber\\
		&=\frac{C}{\beta_1^h}\left(\left(\nu+\frac{\|\textbf{F}\|
			_{H^{-1} (\Omega)}}{\nu}\right)\left\|\bm{\xi}-\bm{\xi}^h\right\|_{\mathbb V}+\left\|p_1-q^h \right\|_{\mathbb W}\right)\label{p^h}.
	\end{align}
	Plugging this into \eqref{p}, we conclude that
	\begin{equation}
		\| p_1-p_1^h\|_{\mathbb W} \le \frac{C}{\beta_1^h}\nu\left(1+\frac{\|\textbf{F}\|
			_{H^{-1} (\Omega)}}{\nu^2}\right)\left\|\bm{\xi}-\bm{\xi}^h\right\|_{\mathbb V}+\left(1+\frac{C}{\beta_1^h}\right)\left\|p_1-q^h \right\|_{\mathbb W}\label{ph_error}.
	\end{equation}
	2)For the error equation \eqref{e_2}, we also decompose the error into the approximate error and the discrete remainder
	\begin{equation}\label{e_9}
		\bm{\eta}-\bm{\eta}^h = \left(\bm{\eta}-I^h\bm{\eta}\right)- \left(\bm{\eta}^h-I^h\bm{\eta}\right)=\bm{\psi}_3-\bm{\psi}_4^h,    \qquad I^h\bm{\eta}\in V^h_{div}.
	\end{equation}
	Inserting this decomposition in the error equation and setting  $\textbf{v}^h=\bm{\psi}_4^h$ leads to
	\begin{align}\label{e_10}
		\nu\left\|\bm{\psi}_4^h \right\|^2_{\mathbb V}= &\nu\mathbb E\left[\left(\nabla\bm{\psi}_3,\nabla\bm{\psi}_4^h\right)\right]+\mathbb E\left[c(\bm{\eta}, \bm{\eta}, \bm{\psi}_4^h)\right]-\mathbb E\left[c(\bm{\eta}^h, \bm{\eta}^h, \bm{\psi}_4^h)\right]+\mathbb E\left[c(\bm{\eta},\bm{\xi}, \bm{\psi}_4^h)\right]\nonumber\\
		&-\mathbb E\left[c(\bm{\eta}^h,\bm{\xi}^h,\bm{\psi}_4^h)\right]+\mathbb E\left[c(\bm{\xi}, \bm{\eta},\bm{\psi}_4^h)\right] -\mathbb E\left[c(\bm{\xi}^h, \bm{\eta}^h,\bm{\psi}_4^h)\right]-\mathbb E\left[(\bm{\psi}_4^h, p_2-q^h)\right].
	\end{align}
	According to the Young's inequality, we obtain
	$$
	\mathbb E\left[(\bm{\psi}_4^h, p_2-q^h)\right]\leq \frac{2}{\nu}\left\|p_2-q^h \right\|^2+ \frac{\nu}{8} \left\|\nabla \bm{\psi}_4^h \right\|^2.
	$$
	Obviously, the estimate for the first item is
	\begin{equation}\label{e_11}
		\nu\mathbb E\left[\left(\nabla\bm{\psi}_3,\nabla\bm{\psi}_4^h\right)\right]\leq 2\nu\left\|\bm{\psi}_3 \right\|^2_{\mathbb V}+ \frac{\nu}{8} \left\|\bm{\psi}_4^h \right\|^2_{\mathbb V}.
	\end{equation}
	The first two trilinear terms are estimated as follows
	\begin{align} \label{e_12}
		&\quad\mathbb E\left[c(\bm{\eta}, \bm{\eta}, \bm{\psi}_4^h)\right]-\mathbb E\left[c(\bm{\eta}^h, \bm{\eta}^h, \bm{\psi}_4^h)\right]\nonumber\\
		&=\mathbb E\left[c\left(\bm{\eta}, \bm{\eta}-\bm{\eta}^h, \bm{\psi}_4^h\right)\right]+\mathbb E\left[c\left(\bm{\eta}-\bm{\eta}^h, \bm{\eta}^h, \bm{\psi}_4^h\right)\right]\nonumber\\
		&=\mathbb E\left[c\left(\bm{\eta}, \bm{\psi}_3, \bm{\psi}_4^h\right)\right]-\mathbb E\left[c\left(\bm{\eta},\bm{\psi}_4^h,\bm{\psi}_4^h\right)\right]+\mathbb E\left[c\left(\bm{\psi}_3,\bm{\eta}^h,\bm{\psi}_4^h\right)\right]-\mathbb E\left[c\left(\bm{\psi}_4^h, \bm{\eta}^h,\bm{\psi}_4^h\right)\right]\nonumber\\
		&=\mathbb E\left[c\left(\bm{\eta}, \bm{\psi}_3,\bm{\psi}_4^h\right)\right]+\mathbb E\left[c\left(\bm{\psi}_3,\bm{\eta}^h,\bm{\psi}_4^h\right)\right]-\mathbb E\left[c\left(\bm{\psi}_4^h,\bm{\eta}^h,\bm{\psi}_4^h\right)\right].
	\end{align}
	For the first term in equation \eqref{e_12}, using \eqref{e_11} yields
	\begin{equation} \label{e_13}
		\begin{split}
			\mathbb E\left[c\left(\bm{\eta}, \bm{\psi}_3,\bm{\psi}_4^h\right)\right]\leq C_{c_9}\left\|\bm{\eta}\right\|_{\mathbb V}\left\|\bm{\psi}_3 \right\|_{\mathbb V} \left\|\bm{\psi}_4^h \right\|_{\mathbb V}  \leq \frac{2C_{c_9}}{\nu}\left\|\bm{\psi}_3 \right\|^2_{\mathbb V}\left\|\bm{\eta}\right\|^2_{\mathbb V}+ \frac{\nu}{8} \left\| \bm{\psi}_4^h \right\|^2_{\mathbb V}. 
		\end{split}
	\end{equation}
	The estimation for the second term of equation \eqref{e_12} is performed analogously, yielding
	\begin{equation}
		\begin{split}
			\mathbb E\left[c\left(\bm{\psi}_3,\bm{\eta}^h,\bm{\psi}_4^h\right)\right] &\leq \frac{2C_{c_{10}}}{\nu}\left\|\bm{\psi}_3 \right\|^2_{\mathbb V}\left\|\bm{\eta}^h\right\|^2_{\mathbb V}+ \frac{\nu}{8} \left\| \bm{\psi}_4^h \right\|^2_{\mathbb V},
		\end{split}
	\end{equation}
	and
	\begin{align}
		&\quad  \mathbb E\left[c\left(\bm{\psi}_4^h,\bm{\eta}^h,\bm{\psi}_4^h\right)\right]\nonumber\\
		&\leq C_{c_{11}}\left\|\bm{\eta}^h\right\|_{\mathbb V}\left\|\bm{\psi}_4^h \right\|^2_{\mathbb V} \nonumber\\
		&\leq C_{c_{11}}\left(\frac{\mathbb E\left[\left\|\sigma\frac{d\textbf W}{d\textbf x}\right\|_{H^{-1}(\Omega)}\right]+\hat{C}_b\| p^h\|_{\mathbb W}}{\nu- \hat{C}_{c_4}\frac{\|\textbf{F}\|
				_{H^{-1} (\Omega)}}{\nu}}\right)\left\|\bm{\psi}_4^h \right\|^2_{\mathbb V}.
	\end{align}
	Similarly, we can calculate
	\begin{align} \label{e1}
		&\quad \mathbb E\left[c\left(\bm{\eta}, \bm{\xi}, \bm{\psi}_4^h\right)\right]- \mathbb E\left[c\left(\bm{\eta}^h, \bm{\xi}^h, \bm{\psi}_4^h\right)\right]\nonumber\\
		&=\mathbb E\left[c\left(\bm{\eta}, \bm{\xi}-\bm{\xi}^h, \bm{\psi}_4^h\right)\right]+\mathbb E\left[c\left(\bm{\eta}-\bm{\eta}^h, \bm{\xi}^h, \bm{\psi}_4^h\right)\right]\nonumber\\
		&=\mathbb E\left[c\left(\bm{\eta}, \bm{\xi}-\bm{\xi}^h, \bm{\psi}_4^h\right)\right]+\mathbb E\left[c\left(\bm{\psi}_3,\bm{\xi}^h,\bm{\psi}_4^h\right)\right]-\mathbb E\left[c\left(\bm{\psi}_4^h, \bm{\xi}^h,\bm{\psi}_4^h\right)\right],
	\end{align}
	where
	\begin{equation}\label{e2}
		\begin{split}
			\mathbb E\left[c\left(\bm{\eta}, \bm{\xi}-\bm{\xi}^h, \bm{\psi}_4^h\right)\right] &\leq \frac{2C_{c_{12}}}{\nu}\left\|\bm{\eta} \right\|^2_{\mathbb V}\left\|\bm{\xi}-\bm{\xi}^h\right\|^2_{\mathbb V}+ \frac{\nu}{8} \left\| \bm{\psi}_4^h \right\|^2_{\mathbb V}, \\ 
			\mathbb E\left[c\left(\bm{\psi}_3,\bm{\xi}^h,\bm{\psi}_4^h\right)\right]&\leq \frac{2C_{c_{13}}}{\nu}\left\|\bm{\psi}_3\right\|^2_{\mathbb V}\left\|\bm{\xi}^h\right\|^2_{\mathbb V}+ \frac{\nu}{8} \left\| \bm{\psi}_4^h \right\|^2_{\mathbb V},\\
			\mathbb E\left[c\left(\bm{\psi}_4^h, \bm{\xi}^h,\bm{\psi}_4^h\right)\right]&\leq C_{c_{14}}\left\|\bm{\xi}^h\right\|_{\mathbb V}\left\|\bm{\psi}_4^h \right\|^2_{\mathbb V}.
		\end{split}
	\end{equation}
	Naturally, we have
	\begin{equation} \label{e3}
		\begin{split}
			\mathbb E\left[c\left(\bm{\xi},\bm{\eta},\bm{\psi}_4^h\right)\right]-\mathbb E\left[c\left(\bm{\xi}^h,\bm{\eta}^h,\bm{\psi}_4^h\right)\right]&=\mathbb E\left[c\left(\bm{\xi}, \bm{\eta}-\bm{\eta}^h, \bm{\psi}_4^h\right)\right]+\mathbb E\left[c\left(\bm{\xi}-\bm{\xi}^h, \bm{\eta}^h, \bm{\psi}_4^h\right)\right]\\
			&=\mathbb E\left[c\left(\bm{\xi}, \bm{\psi}_3,\bm{\psi}_4^h\right)\right]+\mathbb E\left[c\left(\bm{\xi}-\bm{\xi}^h, \bm{\eta}^h, \bm{\psi}_4^h\right)\right],
		\end{split}
	\end{equation}
	where
	\begin{equation}\label{e4}
		\begin{split}
			\mathbb E\left[c\left(\bm{\xi},\bm{\psi}_3,\bm{\psi}_4^h\right)\right]&\leq \frac{2C_{c_{15}}}{\nu}\left\|\bm{\xi}\right\|^2_{\mathbb V}\left\|\bm{\psi}_3\right\|^2_{\mathbb V}+ \frac{\nu}{8} \left\| \bm{\psi}_4^h \right\|^2_{\mathbb V},\\
			\mathbb E\left[c\left(\bm{\xi}-\bm{\xi}^h,\bm{\eta}^h,  \bm{\psi}_4^h\right)\right] &\leq \frac{2C_{c_{16}}}{\nu}\left\|\bm{\xi}-\bm{\xi}^h\right\|^2_{\mathbb V}\left\|\bm{\eta}^h \right\|^2_{\mathbb V}+ \frac{\nu}{8} \left\| \bm{\psi}_4^h \right\|^2_{\mathbb V}.
		\end{split}
	\end{equation}
	
	According to equation \eqref{stronger condition}, substitute the above error estimates into equation \eqref{e_10}, we get
	\begin{equation} 
		\begin{split}
			\left\|\bm{\psi}_4^h \right\|^2_{\mathbb V} &\leq C\left\{\left(1+\frac{N^2_{\eta}}{\nu^2}+\frac{\|\textbf{F}\|^2
				_{H^{-1} (\Omega)}}{\nu^4} \right)\left\|\bm{\psi}_3 \right\|^2_{\mathbb V}+\frac{N^2_{\eta}N^2_{\bm{\xi}-\bm{\xi}^h}}{\nu^2}  \right\}\\
			&\leq C\left\{\left\|\bm{\psi}_3 \right\|^2_{\mathbb V}+\frac{N^2_{\eta}N^2_{\bm{\xi}-\bm{\xi}^h}}{\nu^2}  \right\}.
			\nonumber
		\end{split}
	\end{equation}
	Finally, the error estimation of equation \eqref{e_2} is drawn by using the triangle inequality
	\begin{equation} 
		\begin{split}
			\left\|\bm{\eta}-\bm{\eta}^h\right\|_{\mathbb V}&\leq \left\|\bm{\psi}_3\right\|_{\mathbb V}+\left\|\bm{\psi}_4^h\right\|_{\mathbb V}\\
			&\le C\left\{\left\|\bm{\psi}_3 \right\|_{\mathbb V}+\frac{N_{\eta}N_{\bm{\xi}-\bm{\xi}^h}}{\nu}  \right\}\\
			&\le C\left\{\inf\limits_{\textbf{v}^h \in\textbf{V}^h_{div}}\left\|\bm{\eta}-\textbf{v}^h\right\|_{\mathbb V}+\frac{N_{\eta}N_{\bm{\xi}-\bm{\xi}^h}}{\nu}  \right\}.\nonumber
		\end{split}
	\end{equation}
	On the other hand, we consider the following  finite element error estimate for the pressure.
	Observe that by the triangle inequality, one has
	\begin{equation}
		\|p_2-p_2^h\|_{\mathbb W}\le \|p_2-q^h\|_{\mathbb W} +\|p_2^h-q^h\|_{\mathbb W}.\label{p2}
	\end{equation}
	Next, we handle the term $\|p_2^h-q^h\|_{\mathbb W}$. By the discrete inf-sup condition \eqref{D_LBB} and the insertion of the finite element problem \eqref{D_sSNS2} as well as the variational form of the steady-state Navier-Stokes equations \eqref{w_sSNS2}, we have
	\begin{align*}
		&\quad	\| p_2^h-q^h\|_{\mathbb W}\\
		&\le  \frac{1}{\beta_1^h}\sup_{0\neq\textbf v^h\in \mathbb V^h_0}\frac{\mathbb E[b(\textbf{v}^h, p_2^h-q^h)]}{\|\textbf v^h\|_{\mathbb V^h_0}} \\
		&=\frac{1}{\beta_1^h}\sup_{0\neq\textbf v^h\in \mathbb V^h_0}\frac{\mathbb E[a(\bm{\eta}-\bm{\eta}^h,\textbf v^h)]-\mathbb E[b(\textbf{v}^h, p_2-q^h)]+\mathbb E\left[c(\bm{\eta}, \bm{\eta}, \textbf{v}^h)\right]-\mathbb E\left[c(\bm{\eta}^h, \bm{\eta}^h, \textbf{v}^h)\right]}{\|\textbf v^h\|_{\mathbb V^h_0}}\\
		&\quad\frac{+\mathbb E\left[c(\bm{\eta},\bm{\xi}, \textbf{v}^h)\right]
			-\mathbb E\left[c(\bm{\eta}^h,\bm{\xi}^h,\textbf{v}^h)\right]+\mathbb E\left[c(\bm{\xi}, \bm{\eta},\textbf{v}^h)\right] -\mathbb E\left[c(\bm{\xi}^h, \bm{\eta}^h,\textbf{v}^h)\right]}{\|\textbf v^h\|_{\mathbb V^h_0}}\\
		&=\frac{1}{\beta_1^h}\sup_{0\neq\textbf v^h\in \mathbb V^h_0}\frac{\mathbb E[a(\bm{\eta}-\bm{\eta}^h,\textbf v^h)]-\mathbb E[b(\textbf{v}^h, p_2-q^h)]+\mathbb E[c(\bm{\eta}, \bm{\eta}-\bm{\eta}^h,\textbf v^h)]}{\|\textbf v^h\|_{\mathbb V^h_0}}\\
		&\quad\frac{+\mathbb E[c(\bm{\eta}-\bm{\eta}^h, \bm{\eta}^h,\textbf v^h)]+\mathbb E[c(\bm{\eta}, \bm{\xi}-\bm{\xi}^h,\textbf v^h)]+\mathbb E[c(\bm{\eta}-\bm{\eta}^h, \bm{\xi}^h,\textbf v^h)]+\mathbb E[c(\bm{\xi}, \bm{\eta}-\bm{\eta}^h,\textbf v^h)]}{\|\textbf v^h\|_{\mathbb V^h_0}}\\
		&\quad+\frac{\mathbb E[c(\bm{\xi}-\bm{\xi}^h, \bm{\eta}^h,\textbf v^h)]}{\|\textbf v^h\|_{\mathbb V^h_0}}.
	\end{align*}
	Now,we use the Cauchy-Schwarz inequality
	to yield that
	\begin{align*}
		&\quad	\| p_2^h-q^h\|_{\mathbb W}\\
		&\le  \frac{1}{\beta_1^h}\left(\nu\|\bm{\eta}-\bm{\eta}^h\|_{\mathbb V^h_0}+C\|\bm{\eta}\|_{\mathbb V^h_0}\|\bm{\eta}-\bm{\eta}^h\|_{\mathbb V^h_0}+C\|\bm{\eta}^h\|_{\mathbb V^h_0}\|\bm{\eta}-\bm{\eta}^h\|_{\mathbb V^h_0}+C\|\bm{\eta}\|_{\mathbb V^h_0}\|\bm{\xi}-\bm{\xi}^h\|_{\mathbb V^h_0}\right.\\
		&\left.+C\|\bm{\xi}^h\|_{\mathbb V^h_0}\|\bm{\eta}-\bm{\eta}^h\|_{\mathbb V^h_0}+C\|\bm{\xi}\|_{\mathbb V^h_0}\|\bm{\eta}-\bm{\eta}^h\|_{\mathbb V^h_0}+C\|\bm{\eta}^h\|_{\mathbb V^h_0}\|\bm{\xi}-\bm{\xi}^h\|_{\mathbb V^h_0}+\|p_2-q^h\|_{\mathbb W} \right)   \\
		&=\frac{C}{\beta_1^h}\left(N_{\eta}\left\|\bm{\xi}-\bm{\xi}^h\right\|_{\mathbb V}+\left(\nu+\frac{\|\textbf{F}\|
			_{H^{-1} (\Omega)}}{\nu}+N_{\eta}\right)\left\|\bm{\eta}-\bm{\eta}^h\right\|_{\mathbb V}+\left\|p_2-q^h \right\|_{\mathbb W}\right).
	\end{align*}
	Plugging this into \eqref{p2}, we conclude that
	\begin{align}
		&\quad	\| p_2-p_2^h\|_{\mathbb W}\nonumber \\
		&\le \frac{C}{\beta_1^h} N_{\eta}\left\|\bm{\xi}-\bm{\xi}^h\right\|_{\mathbb V}+\frac{C}{\beta_1^h}\nu\left(\nu+\frac{\|\textbf{F}\|
			_{H^{-1} (\Omega)}}{\nu}+N_{\eta}\right)\left\|\bm{\eta}-\bm{\eta}^h\right\|_{\mathbb V}\nonumber\\
		&\quad+\left(1+\frac{C}{\beta_1^h}\right)\left\|p_2-q^h \right\|_{\mathbb W}\label{p2_error}.
	\end{align}
	Summing up, we finish the proof of \textbf{Theorem \ref{thm4_4}}.\end{proof}
\subsection{Statistical error of the modified splitting format}
Accordingly, the Galerkin finite element discrete formulation of the modified splitting method  is as follows:  Find $\tilde{\bm{\eta}}^h\in \mathbb V^h,\bm{\xi}^h\in \textbf{V}^h , \tilde{p}_2^h\in\mathbb W^h$ such that
\begin{equation}\label{mD_sSNS2}
	\begin{split}
		\mathbb E[a(\tilde{\bm{\eta}}^h,\textbf{w}^h)]  +\mathbb E[c(\tilde{\bm{\eta}}^h,\bm{\xi}^h, \textbf w^h)]+\mathbb E[c(\bm{\xi}^h, \tilde{\bm{\eta}}^h,\textbf{w}^h)]+\mathbb E[b(\textbf w^h, \tilde{p}_2^h)] &= \mathbb E\left[\left(\sigma\frac{d\textbf W}{d\textbf x}, \textbf{w}^h\right)\right],\\
		\mathbb E[b(\tilde{\bm{\eta}}^h, q^h)]&= 0,
	\end{split}
\end{equation}
for all test function $(\textbf{w}^h,q^h)\in\mathbb V^h\times \mathbb Q^h$.
\begin{theorem}[Error Estimation for the Modified Splitting Method] \label{thm4_5}
	Let $\Omega \in \mathbb{R}^2$ be a bounded domain with polyhedral and Lipschitz continuous boundary, let \eqref{Th1} be fulfilled, and let instead of \eqref{ex_solution} the stronger condition  
	\begin{equation}\label{modified condition}
		\frac{\|\textbf{F}\|
			_{H^{-1} (\Omega)}}{\nu^2}\le \frac{5}{8}.    
	\end{equation}
	Meanwhile, $\bm{\xi}^h$, $\tilde{p}_2^h$, $\tilde{\bm{\eta}}^h$ be the unique solution of equation \eqref{D_sSNS1} and \eqref{mD_sSNS2}. Assume that this problem is discretized with inf-sup stable finite element spaces. Then, the following error estimate holds
	\begin{flalign}
		\left\|\tilde{\bm{\eta}}-\tilde{\bm{\eta}}^h\right\|_{\mathbb V} &\le C\left\{\inf\limits_{\textbf{w}^h \in\textbf{V}^h_{div}}\left\|\tilde{\bm{\eta}}-\textbf{w}^h\right\|_{\mathbb V}+\frac{N_{\tilde{\bm{\eta}}}N_{\bm{\xi}-\bm{\xi}^h}}{\nu}  \right\},\\
		\| \tilde{p}_2-\tilde{p}_2^h\|_{\mathbb W} &\le \frac{C}{\beta_1^h}N_{\bm{\eta}}\left\|\bm{\xi}-\bm{\xi}^h\right\|_{\mathbb V}+\frac{C}{\beta_1^h}\left(\nu+\frac{\|\textbf{F}\|
			_{H^{-1} (\Omega)}}{\nu}\right)\left\|\bm{\eta}-\bm{\eta}^h\right\|_{\mathbb V} +\left(1+\frac{C}{\beta_1^h}\right)\left\|\tilde{p}_2-q^h \right\|_{\mathbb W},
	\end{flalign}
	where the constant $N_{\bm{\xi}-\bm{\xi}^h}=C\left\{\left(1+\frac{\|\textbf{F}\|
		_{H^{-1} (\Omega)}}{\nu^2}\right)\inf\limits_{\textbf{w}^h \in\textbf{V}^h_{div}}\left\|\bm{\xi}-\textbf{w}^h\right\|_{\mathbb V}+\frac{1}{\nu}\inf\limits_{q^h \in\textbf{Q}^h}\left\|p-q^h \right\|\right\}$ and the constant C does not depend on the mesh size.
\end{theorem}
\begin{proof} To establish the error estimate, we use test function $\textbf{w}^h \in \textbf{V}^h_{div}$ and subtract the discrete formulation  from
	the weak formulation of the steady-state NS equation
	\begin{gather}
		a(\bm{\xi}-\bm{\xi}^h, \textbf{w}^h)+c(\bm{\xi},\bm{\xi}, \textbf{w}^h) -c(\bm{\xi}^h,\bm{\xi}^h, \textbf{w}^h)= 0, a.s.\label{ie_1} \\
		a(\tilde{\bm{\eta}}-\tilde{\bm{\eta}}^h,\textbf{w}^h)+c(\tilde{\bm{\eta}},\bm{\xi}, \textbf{w}^h)-c(\tilde{\bm{\eta}}^h,\bm{\xi}^h,\textbf{w}^h)+c(\bm{\xi}, \tilde{\bm{\eta}},\textbf{w}^h) -c(\bm{\xi}^h, \tilde{\bm{\eta}}^h,\textbf{w}^h) = 0, a.s.\label{ie_2}
	\end{gather}
	For the error equation \eqref{ie_2}, we also decompose the error into the approximate error and the discrete remainder
	\begin{equation}\label{ie_9}
		\tilde{\bm{\eta}}-\tilde{\bm{\eta}}^h = \left(\tilde{\bm{\eta}}-I^h\tilde{\bm{\eta}}\right)- \left(\tilde{\bm{\eta}}^h-I^h\tilde{\bm{\eta}}\right)=\bm{\psi}_3-\bm{\psi}_4^h,    \qquad I^h\tilde{\bm{\eta}}\in V^h_{div}.
	\end{equation}
	Inserting this decomposition in the error equation and setting  $\textbf{w}^h=\bm{\psi}_4^h$ leads to
	\begin{equation}\label{ie_10}
		\begin{split}
			\nu\left\|\bm{\psi}_4^h \right\|^2_{\mathbb V}= &\nu\mathbb E\left[\left(\nabla\bm{\psi}_3,\nabla\bm{\psi}_4^h\right)\right]+\mathbb E\left[c(\tilde{\bm{\eta}},\bm{\xi}, \bm{\psi}_4^h)\right]-\mathbb E\left[c(\tilde{\bm{\eta}}^h,\bm{\xi}^h,\bm{\psi}_4^h)\right]\\
			&+\mathbb E\left[c(\bm{\xi}, \tilde{\bm{\eta}},\bm{\psi}_4^h)\right] -\mathbb E\left[c(\bm{\xi}^h, \tilde{\bm{\eta}}^h,\bm{\psi}_4^h)\right].
		\end{split}
	\end{equation}
	
	According to equation \eqref{e_11} and \eqref{e1}-\eqref{e4}, substitute the above error estimates into equation \eqref{ie_10}, we get
	\begin{equation} 
		\begin{split}
			\left\|\bm{\psi}_4^h \right\|^2_{\mathbb V} &\leq C\left\{\left(1+\frac{\|\textbf{F}\|^2
				_{H^{-1} (\Omega)}}{\nu^4} \right)\left\|\bm{\psi}_3 \right\|^2_{\mathbb V}+\frac{N^2_{\eta}N^2_{\bm{\xi}-\bm{\xi}^h}}{\nu^2}  \right\}\\
			&\leq C\left\{\left\|\bm{\psi}_3 \right\|^2_{\mathbb V}+\frac{N^2_{\eta}N^2_{\bm{\xi}-\bm{\xi}^h}}{\nu^2}  \right\}.
			\nonumber
		\end{split}
	\end{equation}
	Finally, the error estimation of equation \eqref{ie_9} is drawn by using the triangle inequality
	\begin{equation} 
		\begin{split}
			\left\|\tilde{\bm{\eta}}-\tilde{\bm{\eta}}^h\right\|_{\mathbb V}&\leq \left\|\bm{\psi}_3\right\|_{\mathbb V}+\left\|\bm{\psi}_4^h\right\|_{\mathbb V}\\
			&\le C\left\{\left\|\bm{\psi}_3 \right\|_{\mathbb V}+\frac{N_{\eta}N_{\bm{\xi}-\bm{\xi}^h}}{\nu}  \right\}\\
			&\le C\left\{\inf\limits_{\textbf{w}^h \in\textbf{V}^h_{div}}\left\|\tilde{\bm{\eta}}-\textbf{w}^h\right\|_{\mathbb V}+\frac{N_{\eta}N_{\bm{\xi}-\bm{\xi}^h}}{\nu}  \right\}.\nonumber
		\end{split}
	\end{equation}
	On the other hand, we consider the following  finite element error estimate for the pressure.
	Observe that by the triangle inequality, one has
	\begin{equation}
		\|\tilde{p}_2-\tilde{p}_2^h\|_{\mathbb W}\le \|\tilde{p}_2-q^h\|_{\mathbb W} +\|\tilde{p}_2^h-q^h\|_{\mathbb W}.\label{mp2}
	\end{equation}
	Next, we handle the term $\|\tilde{p}_2^h-q^h\|_{\mathbb W}$. By the discrete inf-sup condition \eqref{D_LBB} and the insertion of the finite element problem \eqref{D_sSNS2} as well as the variational form of the steady-state Navier-Stokes equations \eqref{w_sSNS2}, we have
	\begin{align}
		&\quad\| \tilde{p}_2^h-q^h\|_{\mathbb W}\nonumber\\
		&\le  \frac{1}{\beta_1^h}\sup_{0\neq\textbf w^h\in \mathbb V^h_0}\frac{\mathbb E[b(\textbf{w}^h, \tilde{p}_2^h-q^h)]}{\|\textbf w^h\|_{\mathbb V^h_0}} \nonumber\\
		&=\frac{1}{\beta_1^h}\sup_{0\neq\textbf w^h\in \mathbb V^h_0}\frac{\mathbb E[a(\tilde{\bm{\eta}}-\tilde{\bm{\eta}}^h,\textbf w^h)]-\mathbb E\left[b(\textbf{w}^h, \tilde{p}_2-q^h)\right]+\mathbb E\left[c(\tilde{\bm{\eta}},\bm{\xi}, \textbf{w}^h)\right]}{\|\textbf w^h\|_{\mathbb V^h_0}}\nonumber\\
		&\quad\frac{-\mathbb E\left[c(\tilde{\bm{\eta}}^h,\bm{\xi}^h,\textbf{w}^h)\right]+
			\mathbb E\left[c(\bm{\xi}, \tilde{\bm{\eta}},\textbf{w}^h)\right] -\mathbb E\left[c(\bm{\xi}^h, \tilde{\bm{\eta}}^h,\textbf{w}^h)\right]}{\|\textbf w^h\|_{\mathbb V^h_0}}\nonumber\\ 
		&=\frac{1}{\beta_1^h}\sup_{0\neq\textbf w^h\in \mathbb V^h_0}\frac{\mathbb E[a(\tilde{\bm{\eta}}-\tilde{\bm{\eta}}^h,\textbf w^h)]-\mathbb E[b(\textbf{w}^h, \tilde{p}_2-q^h)]+\mathbb E[c(\tilde{\bm{\eta}}, \bm{\xi}-\bm{\xi}^h,\textbf w^h)]}{\|\textbf w^h\|_{\mathbb V^h_0}}\nonumber\\
		&\quad\frac{+\mathbb E[c(\tilde{\bm{\eta}}-\tilde{\bm{\eta}}^h, \bm{\xi}^h,\textbf w^h)]+\mathbb E[c(\bm{\xi}, \tilde{\bm{\eta}}-\tilde{\bm{\eta}}^h,\textbf w^h)]+\mathbb E[c(\bm{\xi}-\bm{\xi}^h, \tilde{\bm{\eta}}^h,\textbf w^h)]}{\|\textbf w^h\|_{\mathbb V^h_0}}\label{mp2^h1}.
	\end{align}
	Now we use the Cauchy-Schwarz inequality to yield that
	\begin{align}
		&\quad	\| \tilde{p}_2^h-q^h\|_{\mathbb W}\nonumber\\
		&\le  \frac{1}{\beta_1^h}\left(\nu\|\tilde{\bm{\eta}}-\tilde{\bm{\eta}}^h\|_{\mathbb V^h_0}+C\|\tilde{\bm{\eta}}\|_{\mathbb V^h_0}\|\bm{\xi}-\bm{\xi}^h\|_{\mathbb V^h_0}
		+C\|\bm{\xi}^h\|_{\mathbb V^h_0}\|\tilde{\bm{\eta}}-\bm{\eta}^h\|_{\mathbb V^h_0}+C\|\bm{\xi}\|_{\mathbb V^h_0}\|\tilde{\bm{\eta}}-\tilde{\bm{\eta}}^h\|_{\mathbb V^h_0}\right. \nonumber\\	&\quad\left.+C\|\tilde{\bm{\eta}}^h\|_{\mathbb V^h_0}\|\bm{\xi}-\bm{\xi}^h\|_{\mathbb V^h_0}+\|\tilde{p}_2-q^h\|_{\mathbb W}\right)  \nonumber \\
		&=\frac{C}{\beta_1^h}\left(N_{\bm{\eta}}\left\|\bm{\xi}-\bm{\xi}^h\right\|_{\mathbb V}+\left(\nu+\frac{\|\textbf{F}\|
			_{H^{-1} (\Omega)}}{\nu}\right)\left\|\bm{\eta}-\bm{\eta}^h\right\|_{\mathbb V}+\left\|\tilde{p}_2-q^h \right\|_{\mathbb W}\right)\label{mp2_h}.
	\end{align}
	Plugging this into \eqref{mp2}, we conclude that
	\begin{align}
		&\quad\| \tilde{p}_2-\tilde{p}_2^h\|_{\mathbb W}\nonumber \\
		&\le \frac{C}{\beta_1^h}N_{\bm{\eta}}\left\|\bm{\xi}-\bm{\xi}^h\right\|_{\mathbb V}+\frac{C}{\beta_1^h}\left(\nu+\frac{\|\textbf{F}\|
			_{H^{-1} (\Omega)}}{\nu}\right)\left\|\bm{\eta}-\bm{\eta}^h\right\|_{\mathbb V}+\left(1+\frac{C}{\beta_1^h}\right)\left\|\tilde{p}_2-q^h \right\|_{\mathbb W}\label{mp2_error}.
	\end{align}
	Summing up, we finish the proof of \textbf{Theorem \ref{thm4_5}}.
\end{proof}

Based on the conclusion of \textbf{Remark \ref{remark3_2}}, the solution of the deterministic equation in the splitting method for large-scale computations is very close to the expected value $\mathbb E[\textbf{u}]$, so it does not need to be solved repeatedly. Additionally, according to the error conditions of \textbf{Theorem \ref{thm4_4}} and \textbf{Theorem \ref{thm4_5}}, we  conclude that the nonlinear terms in Equation \eqref{D_sSNS2} can be neglected when both Condition \eqref{stronger condition} and Condition \eqref{modified condition} are satisfied.

In large-scale computations, solving nonlinear equations often requires iterative methods, which can be time-consuming and computationally expensive. However, by simplifying the equation to a linear one, we can use faster and more direct methods for solving,  such as matrix operation or analytical solution method, which will significantly speed up the computational efficiency while ensuring accuracy.
\section{Numerical experiments}\label{sec5}
In this section, we conduct numerical tests to validate our theoretical error estimates and evaluate the performance of the proposed splitting method and the modified splitting method. For this purpose, we consider the  steady-state SNS  equation as the underlying problem. The model is constructed based on reference data obtained through the application of the finite element method. All numerical experiments are executed using MATLAB 2021a software. These rigorous tests are undertaken to assess the accuracy and effectiveness of the proposed splitting method, replicating realistic scenarios and mitigating potential sources of bias.

Consider a two-dimensional  steady-state incompressible  SNS equation with Dirichlet boundary conditions
\begin{eqnarray}
	\begin{cases}
		-\nu\Delta \textbf{u}+(\textbf{u}\cdot\nabla) \textbf{u}+\nabla p=F+\sigma\frac{d\textbf W}{d\textbf x} \quad \text{ in } \Omega, \nonumber\\[1ex]
		\nabla\cdot\textbf{u}=0\quad \text{in }\Omega ,\\
		\textbf{u}|_{\partial\Omega}=\textbf{0},
	\end{cases}
\end{eqnarray}
where $\textbf{u}=(u_1,u_2)^{T}$. Specifically, we assume that the domain $\Omega$ is a rectangular region defined as: $\Omega=[0, 1]\times[0, 1]$, and the initial mesh for the triangular finite element is 312. We use the Taylor-Hood finite element method for the computation, and the finite element triangulation of domain $\Omega$ is depicted in  \textbf{FIG.1}.
\begin{figure}[h]
	\centering
	\vspace{-0.3cm}
	\centering
	\includegraphics[height=2.5cm,width=6cm]{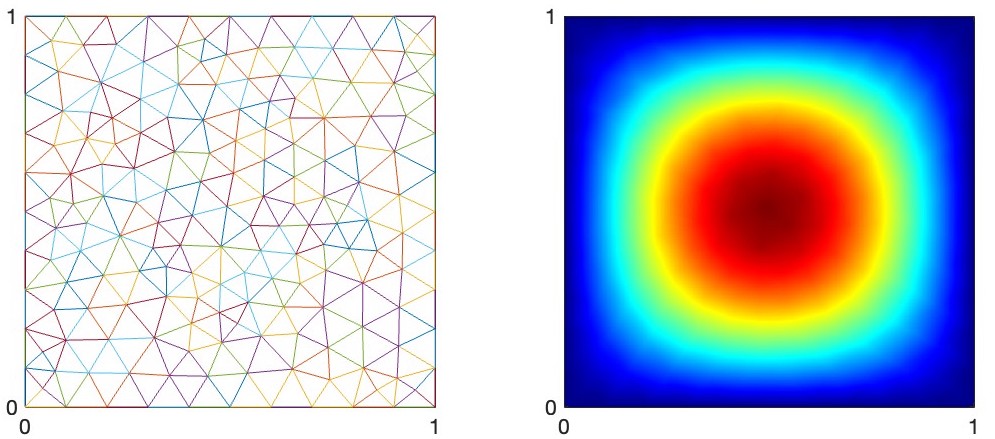}
	\caption{Finite element mesh for domain  $\Omega$}\label{fig2}
\end{figure}

In our numerical test, we assume that when $\sigma=0$, the true solution $\textbf{u}$ and $p$ are as follows
\begin{equation} 
	\begin{split}
		u_1&=64x^2(1-x)^2(4-8y)y(1-y),\\
		u_2&=-64(4-8x)x(1-x)y^2(1-y)^2,\\
		p&=\sin(\pi x)\sin(\pi y).
	\end{split}
\end{equation}
Subsequently, the external  force $\textbf{F}$  is adjusted to align with the true solution $\textbf{u}$, i.e.,
\begin{align}
	F_1=-\nu\Delta u_1+u_1\frac{\partial u_1}{\partial x}+u_2\frac{\partial u_1}{\partial y}+\frac{\partial p}{\partial x}, \\
	F_2=-\nu\Delta u_2+u_1\frac{\partial u_2}{\partial x}+u_2\frac{\partial u_2}{\partial y}+\frac{\partial p}{\partial y}.
\end{align}
Disparities are apparent between the true solution and the finite element sample solution for any randomly chosen input. To address this, we propose the utilization of the Monte Carlo method to estimate the mean of the finite element sample solutions.   Let  $\textbf u_{h,N}$ denotes the finite element sample solution  generated by the random input $\sigma\frac{d\textbf W}{d\textbf x}$  and $\textbf u_{sh,N}$, $\textbf u_{mh,N}$ represent the finite element sample solutions  obtained from the splitting model and the modified splitting model respectively. Naturally, the MC average for s-th order solution samples is defined as
\begin{flalign}
	\vspace{-0.8cm}
	\mathbb E^s_{h,N}[\textbf u]&=\frac{1}{N}\sum\limits_{k=1}^M\textbf u^s_{h,k},\\
	\mathbb E^s_{sh,M}[\textbf u]&=\frac{1}{M}\sum\limits_{k=1}^M\textbf u^s_{sh,k},\\
	\mathbb E^s_{mh,M}[\textbf u]&=\frac{1}{M}\sum\limits_{k=1}^M\textbf u^s_{mh,k}. 
\end{flalign}
We now give the definition of the associated $L^2(\Omega)$-norm error measurement, as shown in \eqref{sh_error}-\eqref{mh_error}
\begin{flalign}
	\bm{\epsilon}_{sh}&=\left\|
	\mathbb E^s_{sh,M}[\textbf u]-\mathbb E^s_{h,M}[\textbf u]\right\|_{L^2(\Omega)}, \label{sh_error} \\
	\bm{\epsilon}_{mh}&=\left\|\mathbb E^s_{mh,M}[\textbf u]-\mathbb E^s_{h,M}[\textbf u]\right\|_{L^2(\Omega)},\label{mh_error} 
\end{flalign}
where the sample size for the Monte-Carlo average
$\mathbb E^s_{h,M}$ is M.

Additionally, we also give the following estimates of relative error statistics
\begin{flalign}\label{re_error}
	\bm{\epsilon}^r_{sh}&=\frac{\left\|
		\mathbb E^s_{sh,M}[\textbf u]-\mathbb E^s_{h,M}[\textbf u]\right\|_{L^2(\Omega)}}{\left\|\mathbb E^s_{h,M}[\textbf u]\right\|_{L^2(\Omega)}}, \\
	\bm{\epsilon}^r_{mh}&=\frac{\left\|
		\mathbb E^s_{mh,M}[\textbf u]-\mathbb E^s_{h,M}[\textbf u]\right\|_{L^2(\Omega)}}{\left\|\mathbb E^s_{h,M}[\textbf u]\right\|_{L^2(\Omega)}}.
\end{flalign}
These error statistics provide quantitative measures of the accuracy and precision of the splitting models's approximation compared to the finite element solution. By computing these error statistics, we  assess the validity of the splitting models and their performance in approximating the finite element solution. 

Specially, FIG.2  provides valuable statistical information on the behavior and properties of the coupled splitting solution and coupled modified splitting solution, respectively, when using M = 100 sample solutions in Monte Carlo simulations. To simplify the test, we represent noise $\sigma\frac{d\textbf W}{d\textbf x}$ by a random function that satisfies the piecewise constant approximation formula \eqref{noise}. In addition, we take $\nu=0.02$, $\sigma=1.5$, and $s=1$. From left to right, the figure displays the velocity fields and the correlation scale temperature field of the deterministic equation (column 1), the stochastic equation (column 2), the modified stochastic equation (column 3) , the splitting solution (column 4) and the modified splitting solution (column 5) respectively, and compare the solution of the splitting method with the finite element sample solution (column 6). It should be noted that the error between $\mathbb E_{h,100}[\textbf u]$ and $\mathbb E_{sh,100}[\textbf u]$ is very small, only $9.4604\times10^{-16}$, which numerically validates the efficacy of the splitting method. Due to the neglect of the non-linear term of the stochastic equation in the splitting method, the error between $\mathbb E_{mh,100}[\textbf u]$ and $\mathbb E_{h,100}[\textbf u]$ is greater than that between $\mathbb E_{sh,100}[\textbf u]$ and $\mathbb E_{h,100}[\textbf u]$, approximately $5.5304\times10^{-4}$. Furthermore, it can be observed that the shape of the solution to the deterministic equation ($\bm{\xi}_h$, column 1) is closely similar to that of the original equation in terms of the expected value. This observation numerically verifies \textbf{Remark \ref{remark3_2}}.
\begin{figure}[h!]
	\centering
	\includegraphics[height=5cm,width=13cm]{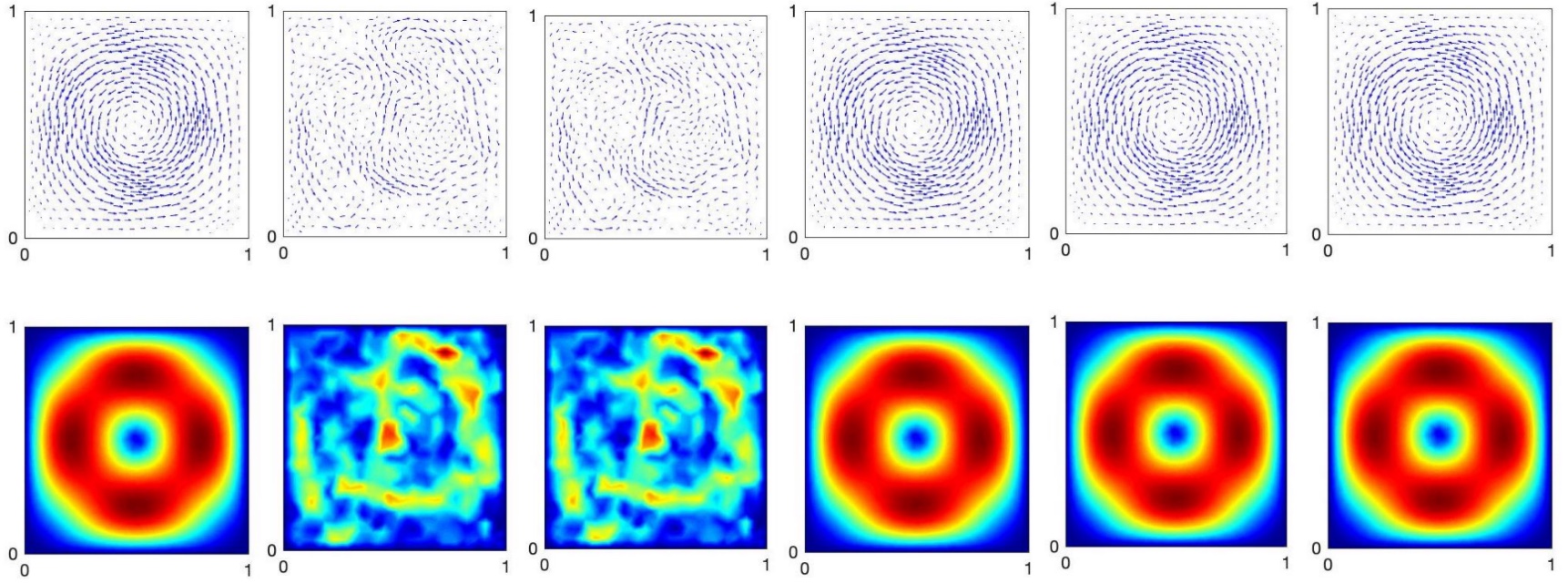}
	\caption{Quiver plot corresponding to  velocity field $\bm{\xi}_h,$ $\mathbb E[\bm \eta_h], \mathbb E[\bm{\hat{\eta}}_h]$, $\mathbb E_{sh,100}[\textbf u], \mathbb E_{mh,100}[\textbf u], \mathbb E_{h,100}[\textbf u]$ (line 1) and the associated  scaled temperature field $|\bm{\xi}_h|,\left|\mathbb E[\bm \eta_h]\right|, \left|\mathbb E[\bm {\hat{\eta}}_h]\right|,$ $\left|\mathbb E_{sh,100}[\textbf u]\right|$, $\left|\mathbb E_{mh,100}[\textbf u]\right|$,$\left|\mathbb E_{h,100}[\textbf u]\right|$(line 2) .}
\end{figure}

To demonstrate the effectiveness of the splitting method under different disturbance conditions, we selected four perturbation rates $\kappa$ = 0.2019, 0.5139, 0.8547, and 1.7118 for numerical experiments, as illustrated in Fig. 3. In particular, the perturbation rate represents the ratio between the given body force $F$ and the additive spatial random noise $\sigma\frac{d\textbf W}{d\textbf x}$, that is,

\begin{equation*}
	\kappa=\frac{\left\|\sigma\frac{d\textbf W}{d\textbf x}\right\|}{\|F\|}.
\end{equation*}
Figure 3 depicts the morphological graphs of the individual solution. In particular, we do not alter the external force term \(  \mathbf{F} \). Consequently, the solution of the deterministic equation within the splitting method remains unchanged and its shape aligns with that shown in the first column of FIG.2. In FIG.3, it is evident that as the perturbation rate increases, the influence of random noise becomes more pronounced, leading to a more irregular shape of the solution. However, the morphology of the solutions obtained from the splitting method and the modified splitting method remains very close to the solution of the original equation (as shown in columns 3-5 in Fig.3).
\begin{figure}[h!]
	\centering
	\centering
	\includegraphics[height=10cm,width=13cm]{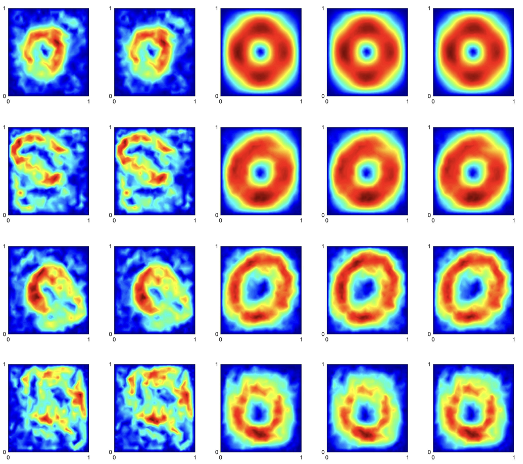}\caption{From left to right, there are  stochastic solution $\bm{\eta}_h$, modified stochastic solution $\hat{\bm{\eta}}_h$, finite element approximate velocity field based on splitting method $\textbf{u}_{sh}$, finite element approximate velocity field based on modified splitting method $\textbf{u}_{mh}$ and  finite element approximate velocity field $\textbf{u}_{h}$. From top to bottom, the scaled temperature fields correspond to $\kappa$=0.2019,  0.5139,  0.8547  and  1.7118 .}
	\vspace{-0.4cm}
\end{figure}

In Table 1, we vary the number of samples ($M$) and calculate the absolute and relative errors between the average solution obtained through the splitting method, modified splitting method and the average solution obtained through finite element sampling. Here, the absolute error represents the difference between the two average solutions, while the relative error is the ratio between the absolute error and the magnitude of the reference solution. As shown in Table 1, both absolute and relative errors decrease as the sample size increases. At the same time, when the number of samples does not need to be very large, the order of magnitude of error tends to be stable, that is, the splitting solution is very close to the finite element solution.The results demonstrate the effectiveness of the splitting method in achieving accurate solutions while reducing computational costs.
\begin{table}[h]
	\vspace{-0.3cm}
	\centering
	\caption{The error measurements $\bm{\epsilon}_{sh}$,$\bm{\epsilon}_{mh}$,$\bm{\epsilon}^r_{sh}$ and $\bm{\epsilon}^r_{mh}$ as $M=50,100,150,200,300,500$, where $\nu=0.02$ and $\kappa=0.51$.}
	\vspace{0.3cm}
	\begin{tabular*}{\textwidth}{@{\extracolsep{\fill}}llccccc}
		\hline\hline $M$ & \quad 50 & 100 & 150 & 200 & 300 & 500 \\
		\hline $\bm{\epsilon}_{sh}$ &$1.26e{-15}$ & $1.13e{-15}$ & 
		$9.8e{-16}$&
		$9.76e{-16}$ & $9.56e{-16}$ & 
		$ 1.02e{-15}$   \\
		$\bm{\epsilon}_{mh}$ &$1.5e{-3}$ & $1.4e{-3}$ & 
		$1.2e{-3}$&
		$1e{-3}$ & $1.1e{-3}$ & $1.1e{-3}$   \\    $\bm{\epsilon}^r_{sh}$ &$1.28e{-15}$ & $1.11e{-15}$ & 
		$9.92e{-16}$&
		$9.84e{-16}$ & $9.67e{-16}$ & 
		$1.02e{-15}$   \\
		$\bm{\epsilon}^r_{mh}$ &$1.6e{-3}$ & $1.4e{-3}$ & 
		$1.2e{-3}$&
		$1e{-3}$ & $1.1e{-3}$ & $1.1e{-3}$  \\
		\hline\hline
	\end{tabular*}
	\vspace{-0.3cm}
\end{table}

Further more, we present the absolute and relative errors between the average solution obtained through two splitting methods and the average solution obtained through finite element sampling in Table 2 for different perturb ratio $\kappa$. In particular, we note that the effectiveness of the splitting method does not depend on the initial value. For the original SNS equation, if the initial point is not selected near the true solution $\textbf{u}$, the solution of the equation becomes unstable when the value of $\sigma$ approaches 8.1. However, simultaneously, the splitting method can still be solved efficiently, that is, the order of error remains at $10^{-15}$. This result also fully indicates numerically that the splitting method expands the stability condition of the equation solution compared with the direct solution.
\begin{table}[h]
	\vspace{-0.3cm}
	\centering
	\caption{{The error measurements  $\bm{\epsilon}_{sh}$, $\bm{\epsilon}_{mh}$,  $\bm{\epsilon}^r_{sh}$ and $\bm{\epsilon}^r_{mh}$ as perturb ratio $\kappa=0.1711, 0.3424, 0.5139, 0.6839, 0.8547
			$, where $\nu=0.02$ and
			$M=200$.}}
	\vspace{0.3cm}
	\begin{tabular*}{\textwidth}{@{\extracolsep{\fill}}llccccc}
		\hline\hline	$\sigma$ & \quad0.8 & 1.6 & 2.4 & 3.2 & 4 &8 \\
		\hline	$\bm{\epsilon}_{sh}$ &$9.39e{-16}$ & $9.59e{-16}$ & 
		$9.76e{-16}$&
		$9.96e{-16}$ & $1.02e{-15}$ & 
		$ 1.17e{-15}$   \\
		$\bm{\epsilon}_{mh}$ &$1.13e{-4}$ & $4.34e{-4}$ & 
		$1e{-3}$&
		$2.1e{-3}$ & $2.9e{-3}$ & $1.7e{-2}$   \\
		$\bm{\epsilon}^r_{sh}$ &$9.45e{-16}$ & $9.66e{-16}$ & 
		$9.84e{-16}$&
		$1e{-15}$ & $1.04e{-15}$ & 
		$ 1.17e{-15}$   \\
		$\bm{\epsilon}^r_{mh}$ &$1.13e{-4}$ & $4.37e{-4}$ & 
		$1e{-3}$&
		$2.1e{-3}$ & $2.9e{-3}$ & $1.69e{-2}$  \\
		\hline\hline
	\end{tabular*}
	\vspace{-0.3cm}
\end{table}


\section{Conclusions and discussions} \label{sec6}
In recent years, the splitting method has garnered significant attention in the field of stochastic partial differential equations (SPDEs) due to its outstanding performance. In this study, we apply the splitting method to  steady-state SNS equations with spatial  noise, effectively decomposing them into a deterministic steady-state NS equation and a stochastic equation.

The stability, existence, and uniqueness of the solutions derived through the splitting method are thoroughly evaluated. By analyzing the properties of the splitting equation, we are able to obtain  more relaxed stability conditions and thus better handle data with nonlinear characteristics. In large-scale calculation, the deterministic equations in the splitting method do not need to be solved repeatedly, and the nonlinear term is ignored under certain conditions, and the modified splitting scheme is adopted, which speeds up the calculation efficiency while ensuring the accuracy.

Moreover, we compute error statistics both theoretically and numerically to assess the accuracy of the splitting method and the modified splitting format. These error statistics provide valuable insights into the quality of the approximate solutions obtained through the splitting method. By comparing the solutions derived via the splitting method with the original equations, we can quantitatively evaluate the performance of the splitting method and validate its accuracy. In the case of small perturbations, the deterministic solutions derived from the splitting method can quickly provide statistical information about the original SNS equations.

In summary, our findings demonstrate that the splitting method serves as an effective tool for the mathematical analysis and numerical approximation of stochastic steady-state NS equations. The theoretical analysis showcases that, by considering the existence and uniqueness of the solutions and the error range, the splitting method can yield precise approximations for the system solutions. 

We emphasize that the proposed splitting method can be extended to nonlinear SPDE problems with driven noise in the future. The advantages of this generalization are mainly reflected in extending the stability of the solution and improving the efficiency of large-scale calculations.As a concrete demonstration, we have successfully implemented this framework to unsteady stochastic Navier-Stokes systems, obtaining promising results.These developments are anticipated to stimulate innovative research pathways in stochastic flow modeling and uncertainty quantification.

\section*{Acknowledgments}
This work of Ju Ming is supported by the National Natural Science Foundation of China (No.12471379).

\section*{Conflicts of Interest} The authors declare no conflicts of interest.


\begin{thebibliography}{00}
	\bibitem[Navier(1823)]{C_L_1823}
	Navier C-L,
	\textit{Mémoire sur les lois du mouvement des fluides},
	Mémoires de l’Académie Royale des Sciences de l’Institut de France,
	1823,
	6: 389-440.
	
	\bibitem[Przemyslaw(2019)]{Przemyslaw_2019}
	Przemyslaw M,
	\textit{Numerical optimization methods comparison based on the CFD conduction-convection heat transfer case},
	International Journal of Numerical Methods for Heat \& Fluid Flow,
	2019,
	29(6): 2080-2092.
	
	\bibitem[Bhatti et al.(2020)]{Bhatti_2020}
	Bhatti MM, Marin M, Zeeshan A, Abdelsalam SI,
	\textit{Recent trends in computational fluid dynamics},
	Frontiers in Physics,
	2020,
	8: 593111.
	
	\bibitem[Yang et al.(2021)]{Yang_2021}
	Yang S, Wang J, Dai G, Yang F, Huang FJ,
	\textit{Controlling macroscopic heat transfer with thermal metamaterials: theory, experiment and application},
	Physics Reports,
	2021,
	908(3): 1-65.
	
	\bibitem[Holden et al.(1996)]{Holden_1996}
	Holden H, Oksendal B, Uboe J, Zhang T,
	\textit{Stochastic Partial Differential Equations: A Modeling, White Noise Approach},
	Birkhäuser,
	Boston,
	1996.
	
	\bibitem[Carmona \& Rozovskii(1999)]{Carmona_1999}
	Carmona R, Rozovskii B,
	\textit{Stochastic Partial Differential Equations: Six Perspectives},
	AMS,
	Providence,
	1999.
	
	\bibitem[Kallianpur \& Xiong(1995)]{Kallianpur_1995}
	Kallianpur G, Xiong J,
	\textit{Stochastic Differential Equations in Infinite Dimensional Spaces},
	Institute of Mathematical Statistics,
	Hayward, California,
	1995.
	
	\bibitem[Lions \& Souganidis(2000)]{Lions_2000}
	Lions P, Souganidis P,
	\textit{Fully nonlinear stochastic pde with semilinear stochastic dependence},
	C.R. Acad. Sci. Paris Sér. I Math.,
	2000,
	331(8): 617-624.
	
	\bibitem[Da Prato \& Zabczyk(1992)]{DaPrato_1992}
	Da Prato G, Zabczyk J,
	\textit{Stochastic Equations in Infinite Dimensions},
	Cambridge University Press,
	Cambridge, UK,
	1992.
	
	\bibitem[Alos \& Bonaccorsi(2002)]{Alos_2002}
	Alós E, Bonaccorsi S,
	\textit{Stochastic Partial Differential Equations with Dirichlet White-Noise Boundary Conditions},
	Ann. I. H. Poincar\'e,
	2002,
	38: 152-154.
	
	\bibitem[Gunzburger et al.(2016)]{Gunzburger_2016}
	Gunzburger MD, Hou L, Ming J,
	\textit{Stochastic steady-state Navier–Stokes equations with additive random noise},
	Journal of Scientific Computing,
	2016,
	66: 672-691.
	
	\bibitem[Deugoue et al.(2021)]{Deugoue_2021}
	Deugoue G, Moghomye BJ, Medjo TT,
	\textit{Splitting-up scheme for the stochastic Cahn–Hilliard Navier–Stokes model},
	Stochastics and Dynamics,
	2021,
	21(1): 2150005.
	
	\bibitem[Bessaih \& Millet(2019)]{Bessaih_2019}
	Bessaih H, Millet A,
	\textit{Strong convergence of time numerical schemes for the stochastic two-dimensional Navier–Stokes equations},
	IMA Journal of Numerical Analysis,
	2019,
	39(4): 2135-2167.
	
	\bibitem[Zhao(2017)]{Zhao_2017}
	Zhao W,
	\textit{Comparison Of Different Noise Forcings, Regularization Of Noise, And Optimal Control For The Stochastic Navier-Stokes Equation},
	PhD thesis,
	The Florida State University,
	2017.
	
	\bibitem[Deugoue \& Sango(2014)]{Deugoue_2014}
	Deugoue G, Sango M,
	\textit{Convergence for a splitting-up scheme for the 3D stochastic Navier-Stokes-$\alpha$ model},
	Stochastic Analysis and Applications,
	2014,
	32(2): 253-279.
	
	\bibitem[Dörsek(2012)]{Dorsek_2012}
	Dörsek P,
	\textit{Semigroup splitting and cubature approximations for the stochastic Navier–Stokes equations},
	SIAM Journal on Numerical Analysis,
	2012,
	50(2): 729-746.
	
	\bibitem[Zhao \& Gunzburger(2020)]{Zhao_2020}
	Zhao W, Gunzburger M,
	\textit{Auxiliary equations approach for the stochastic unsteady Navier–Stokes equations with additive random noise},
	Numer. Math. Theory Methods Appl.,
	2020,
	13: 1-26.
	
	\bibitem[Capdeville(2022)]{Capdeville_2022}
	Capdeville G,
	\textit{Gas kinetic principles in Navier–Stokes finite-volume solvers},
	Journal of Computational Science,
	2022,
	63: 101756.
	
	\bibitem[Lucca et al.(2023)]{Lucca_2023}
	Lucca A, Busto S, Dumbser M,
	\textit{An implicit staggered hybrid finite volume/finite element solver for the incompressible Navier-Stokes equations},
	East Asian Journal on Applied Mathematics,
	2023, 13(3): 671-716.
	
	\bibitem[Cherfils \& Petcu(2016)]{Cherfils_2016}
	Cherfils L, Petcu M,
	\textit{On the viscous Cahn-Hilliard-Navier-Stokes equations with dynamic boundary conditions},
	Communications on Pure \& Applied Analysis,
	2016,
	15(4): 1147-1167.
	
	\bibitem[Cho et al.(2023)]{Cho_2023}
	Cho MH, Ha ST, Yoo JY, et al,
	\textit{A numerical experiment on the stability and convergence characteristics of some splitting mixed-finite element methods for solving the incompressible Navier-Stokes equations},
	Journal of Mechanical Science and Technology,
	2023.
	
	\bibitem[Sousedík \& Elman(2016)]{Sousedík_2016}
	Sousedík B, Elman HC,
	\textit{Stochastic Galerkin methods for the steady-state Navier–Stokes equations},
	Journal of Computational Physics,
	2016,
	316: 435-452.
	
	\bibitem[Strang(1968)]{Strang_1968}
	Strang G,
	\textit{On the construction and comparison of difference schemes},
	SIAM J. Numer. Analysis,
	1968,
	5: 506-517.
	
	\bibitem[Marchuk(1968)]{Marchuk_1968}
	Marchuk GI,
	\textit{Some applications of splitting-up methods to the solution of mathematical physics problems},
	Aplikace Matematiky,
	1968,
	13: 103-132.
	
	\bibitem[Seydaoglu et al.(2016)]{Seydaoglu_2016}
	Seydaoğlu M, Erdoğan U, Öziş T,
	\textit{Numerical solution of Burgers’ equation with high order splitting methods},
	Journal of Computational and Applied
	Mathematics,
	2016,
	291: 410-421.
	
	\bibitem[Bréhier \& Goudenège(2016)]{Brehier_2016}
	Bréhier CE, Goudenège L,
	\textit{Analysis of some splitting schemes for the stochastic Allen-Cahn equation},
	Journal of Computational and Applied Mathematics,
	2016,
	291: 410-421.
	
	\bibitem[Gunzburger et al.(2013)]{Gunzburger_2013}
	Gunzburger MD, Hou L, Ming J,
	\textit{Efficient numerical methods for stochastic partial differential equations through transformation to equations driven by correlated noise},
	International Journal for Uncertainty Quantification,
	2013,3(4).
	
	\bibitem[Chorin(1968)]{Chorin_1968}
	Chorin AJ,
	\textit{Numerical solution for the Navier-Stokes equations},
	Math. Comp.,
	1968,
	22: 745-762.
	
	\bibitem[Temam(1969)]{Temam_1969}
	Temam R,
	\textit{Sur l’approximation de la solution des équations de Navier-Stokes par la méthode des pas fractionnaires (II)},
	Arch. Ration. Mech. Anal.,
	1969,33: 377-385.
	
	\bibitem[Guermond et al.(2006)]{Guermond_2006}
	Guermond J, Minev P, Shen J,
	\textit{An overview of projection methods for incompressible flows},
	Comput. Methods Appl. Mech. Engrg.,
	2006, 195: 6011-6045.
	
	\bibitem[Brown et al.(2001)]{Brown_2001}
	Brown D, Cortez R, Minion M,
	\textit{Accurate projection methods for the incompressible Navier-Stokes equations},
	J. Comput. Phys.,
	2001,168: 464-499.
	
	\bibitem[Prohl(2009)]{Prohl_2009}
	Prohl A,
	\textit{On pressure approximation via projection methods for nonstationary incompressible Navier–Stokes equations},
	SIAM J. Numer. Anal.,
	2009,47: 158-180.
	
	\bibitem[Olshanskii et al.(2010)]{Olshanskii_2010}
	Olshanskii MA, Sokolov A, Turek S,
	\textit{Error analysis of a projection method for the Navier-Stokes equations with Coriolis force},
	J. Math. Fluid Mech.,
	2010,12: 485-502.
	
	\bibitem[Saleri \& Veneziani(2006)]{Saleri_2006}
	Saleri F, Veneziani A,
	\textit{Pressure correction algebraic splitting methods for the incompressible Navier–Stokes equations},
	SIAM J. Numer. Anal.,
	2006,43: 174-194.
	
	\bibitem[Badia et al.(2008)]{Badia_2008}
	Badia S, Quaini A, Quarteroni A,
	\textit{Splitting methods based on algebraic factorization for fluid-structure interaction},
	SIAM J. Sci. Comput.,
	2008, 30: 1778-1805.
	
	\bibitem[Bertagna et al.(2016)]{Bertagna_2016}
	Bertagna L, Quaini A, Veneziani A,
	\textit{Deconvolution-based nonlinear filtering for incompressible flows at moderately large Reynolds numbers},
	Internat. J. Num. Methods Fluids,
	2016, 81: 463-488.
	
	\bibitem[Reynolds(1883)]{Reynolds_1883}
	Reynolds O,
	\textit{An experimental investigation of the circumstances which determine whether the motion of water shall be direct or sinuous, and of the law of resistance in parallel channels},
	Philosophical Transactions of the Royal Society,
	1883,174: 935-982.
	
	\bibitem[Lamb(1895)]{Lamb_1895}
	Lamb H,
	\textit{On the Dynamical Theory of Incompressible Viscous Fluids and the Determination of the Criterion},
	Proceedings of the Royal Society of London,
	1895,
	58: 51-63.
	
	\bibitem[Farrell \& Ioannou(1993)]{Farrell_1993}
	Farrell BF, Ioannou PJ,
	\textit{Stochastic Forcing of the Linearized Navier-Stokes Equations},
	Physics of Fluids A,
	1993,5: 2600-2609.
	
	\bibitem[Jovanović \& Bamieh(2001)]{Jovanovic_2001}
	Jovanović MR, Bamieh B,
	\textit{The Spatio-Temporal Impulse Response of the Linearized Navier-Stokes Equations},
	Proceedings of the 2001 American Control Conference,
	2001: 1948-1953.
	
	\bibitem[Barbu(2011)]{Barbu_2011}
	Barbu V,
	\textit{The internal stabilization by noise of the linearized Navier-Stokes equation},
	ESAIM: Control, Optimisation and Calculus of Variations,
	2011,17(1): 117-130.
	
	\bibitem[Grenier \& Nguyen(2022)]{Grenier_2022}
	Grenier E, Nguyen TT,
	\textit{Green function for linearized Navier–Stokes around a boundary shear layer profile for long wavelengths},
	Annales de l'Institut Henri Poincaré C,2022, 40(6): 1457-1485.
	
	\bibitem[Adams(1978)]{Adams_1978}
	Adams RA,
	\textit{Sobolev spaces},
	Academic Press,
	New York,
	1978.
	
	\bibitem[Gunzburger(1989)]{Gunzburger_1989}
	Gunzburger M,
	\textit{Finite Element Methods for Viscous Incompressible Flows: A Guide to Theory, Practice, and Algorithms},
	Academic Press,
	Boston,
	1989.
	
	\bibitem[Smolyak(1963)]{Smolyak_1963}
	Smolyak S,
	\textit{Quadrature and interpolation formulas for tensor products of certain classes of functions},
	Soviet Mathematics Doklady,
	1963,
	4: 240-243.
	
	\bibitem[Babuška et al.(2004)]{Babuska_2004}
	Babuška I, Tempone R, Zouraris GE,
	\textit{Galerkin finite element approximations of stochastic elliptic partial differential equations},
	SIAM J. Numer. Anal.,
	2004,
	42(2): 800-825.
	
	\bibitem[Gyöngy \& Nualart(1997)]{Gyongy_1997}
	Gyöngy I, Nualart D,
	\textit{Implicit scheme for stochastic parabolic partial differential equations driven by space-time white noise},
	Potential Anal.,
	1997,
	7(4): 725-757.
	
	\bibitem[Ming \& Gunzburger(2013)]{Ming_2013}
	Ming J, Gunzburger M,
	\textit{Efficient numerical methods for stochastic partial differential equations through transformation to equations driven by correlated noise},
	International Journal for Uncertainty Quantification,
	2013.
	
	\bibitem[Gunzburger \& Zhao(2019)]{Gunzburger_2019}
	Gunzburger M, Zhao W,
	\textit{Descriptions, Discretizations, and Comparisons of Time/Space Colored and White Noise Forcings of the Navier-Stokes Equations},
	SIAM Journal on Scientific Computing,
	2019,
	41(4): 2579-2602.
	
	\bibitem[Nobile et al.(2008)]{Nobile_2008}
	Nobile F, Tempone R, Webster CG,
	\textit{An anisotropic sparse grid stochastic collocation method for partial differential equations with random input data},
	SIAM J. Numerical Analysis,
	2008,
	46(5): 2411-2442.
	
	\bibitem[Allen et al.(1998)]{Allen_1998}
	Allen EJ, Novosel SJ, Zhang Z,
	\textit{Finite element and difference approximation of some linear stochastic partial differential equations},
	Stochastics and stochastics reports,
	1998,
	64(1-2): 117-142.
	
	\bibitem[Du \& Zhang(2002)]{Du_2002}
	Du Q, Zhang T,
	\textit{Numerical approximation of some linear stochastic partial differential equations driven by special additive noises},
	SIAM Journal on Numerical Analysis,
	2002,
	40(4): 1421-1445.
	
	\bibitem[Yan(2004)]{Yan_2004}
	Yan Y,
	\textit{Semidiscrete Galerkin approximation for a linear stochastic parabolic partial differential equation driven by an additive noise},
	BIT Numerical Mathematics,
	2004,
	44(4): 829-847.
	
	\bibitem[Temam(2001)]{Temam_2001}
	Temam R,
	\textit{Navier-Stokes Equations, Theory and Numerical Analysis},
	AMS, Chelsea,
	2001.
	
	\bibitem[Girault \& Raviart(1986)]{Girault_1986}
	Girault V, Raviart P-A,
	\textit{Finite Element Methods for Navier-Stokes Equations},
	Springer,
	New York,
	1986.
	
	\bibitem[Douglas et al.(1988)]{Douglas_1988}
	Douglas AN, Scott LR, Vogelius M,
	\textit{Regular inversion of the divergence operator with Dirichlet boundary conditions on a polygon},
	Ann. Scuola Norm. Sup. Pisa Cl. Sci.,
	1988,
	15(2): 169-192.
	
	\bibitem[Hou et al.(2006)]{Hou_2006}
	Hou TY, Luo W, Rozovskii B, Zhou H,
	\textit{Wiener chaos expansions and numerical solutions of randomly forced equations of fluid mechanics},
	J. Comput. Phys.,
	2006,
	216: 687-706.
	
	\bibitem[Powell \& Silvester(2012)]{Powell_2012}
	Powell CE, Silvester DJ,
	\textit{Preconditioning steady-state Navier–Stokes equations with random data},
	SIAM J. Sci. Comput.,
	2012,
	34(5): 2482-2506.
	
\end{thebibliography}
\end{document}